\PassOptionsToPackage{unicode}{hyperref}
\PassOptionsToPackage{hyphens}{url}
\PassOptionsToPackage{dvipsnames,svgnames,x11names}{xcolor}
\documentclass[12pt]{article}
\usepackage{latexsym,amssymb, epsfig, amsmath,amsfonts, amsthm}

\usepackage{yfonts}
\usepackage{iftex}
\ifPDFTeX
  \usepackage[T1]{fontenc}
  \usepackage[utf8]{inputenc}
  \usepackage{textcomp} 
\else 
  \usepackage{unicode-math}
  \defaultfontfeatures{Scale=MatchLowercase}
  \defaultfontfeatures[\rmfamily]{Ligatures=TeX,Scale=1}
\fi
\usepackage{lmodern}
\usepackage{mathrsfs}
\usepackage{booktabs,threeparttable,multirow,siunitx,array,adjustbox}

\usepackage[normalem]{ulem}

\usepackage[linesnumbered, ruled, vlined]{algorithm2e}
\SetKwRepeat{Do}{do}{while}%

\ifPDFTeX\else  
\fi
\IfFileExists{upquote.sty}{\usepackage{upquote}}{}
\IfFileExists{microtype.sty}{
  \usepackage[]{microtype}
  \UseMicrotypeSet[protrusion]{basicmath} 
}{}
\makeatletter
\@ifundefined{KOMAClassName}{
  \IfFileExists{parskip.sty}{%
    \usepackage{parskip}
  }{
    \setlength{\parindent}{0pt}
    \setlength{\parskip}{6pt plus 2pt minus 1pt}}
}{
  \KOMAoptions{parskip=half}}
\makeatother
\usepackage{xcolor}
\makeatletter
\ifx\paragraph\undefined\else
  \let\oldparagraph\paragraph
  \renewcommand{\paragraph}{
    \@ifstar
      \xxxParagraphStar
      \xxxParagraphNoStar
  }
  \newcommand{\xxxParagraphStar}[1]{\oldparagraph*{#1}\mbox{}}
  \newcommand{\xxxParagraphNoStar}[1]{\oldparagraph{#1}\mbox{}}
\fi
\ifx\subparagraph\undefined\else
  \let\oldsubparagraph\subparagraph
  \renewcommand{\subparagraph}{
    \@ifstar
      \xxxSubParagraphStar
      \xxxSubParagraphNoStar
  }
  \newcommand{\xxxSubParagraphStar}[1]{\oldsubparagraph*{#1}\mbox{}}
  \newcommand{\xxxSubParagraphNoStar}[1]{\oldsubparagraph{#1}\mbox{}}
\fi
\makeatother

\usepackage{longtable,booktabs,array}
\usepackage{calc} 
\usepackage{etoolbox}
\usepackage{enumitem}
\makeatletter
\patchcmd\longtable{\par}{\if@noskipsec\mbox{}\fi\par}{}{}
\makeatother
\IfFileExists{footnotehyper.sty}{\usepackage{footnotehyper}}{\usepackage{footnote}}
\makesavenoteenv{longtable}
\usepackage{graphicx}
\makeatletter
\def\maxwidth{\ifdim\Gin@nat@width>\linewidth\linewidth\else\Gin@nat@width\fi}
\def\maxheight{\ifdim\Gin@nat@height>\textheight\textheight\else\Gin@nat@height\fi}
\makeatother

\setkeys{Gin}{width=\maxwidth,height=\maxheight,keepaspectratio}
\makeatletter
\def\fps@figure{htbp}
\makeatother

\makeatletter
\@ifpackageloaded{caption}{}{\usepackage{caption}}
\AtBeginDocument{%
\ifdefined\contentsname
  \renewcommand*\contentsname{Table of contents}
\else
  \newcommand\contentsname{Table of contents}
\fi
\ifdefined\listfigurename
  \renewcommand*\listfigurename{List of Figures}
\else
  \newcommand\listfigurename{List of Figures}
\fi
\ifdefined\listtablename
  \renewcommand*\listtablename{List of Tables}
\else
  \newcommand\listtablename{List of Tables}
\fi
\ifdefined\figurename
  \renewcommand*\figurename{Figure}
\else
  \newcommand\figurename{Figure}
\fi
\ifdefined\tablename
  \renewcommand*\tablename{Table}
\else
  \newcommand\tablename{Table}
\fi
}
\@ifpackageloaded{float}{}{\usepackage{float}}
\floatstyle{ruled}
\@ifundefined{c@chapter}{\newfloat{codelisting}{h}{lop}}{\newfloat{codelisting}{h}{lop}[chapter]}
\floatname{codelisting}{Listing}

\makeatother
\makeatletter
\@ifpackageloaded{caption}{}{\usepackage{caption}}
\@ifpackageloaded{subcaption}{}{\usepackage{subcaption}}
\makeatother

\ifLuaTeX
  \usepackage{selnolig}  
\fi
\usepackage[]{natbib}
\usepackage{bookmark}

\IfFileExists{xurl.sty}{\usepackage{xurl}}{} 
\hypersetup{
  pdftitle={Title},
  pdfauthor={Author 1; Author 2},
  pdfkeywords={3 to 6 keywords, that do not appear in the title},
  colorlinks=true,
  linkcolor={blue},
  filecolor={Maroon},
  citecolor={Blue},
  urlcolor={Blue},
  pdfcreator={LaTeX via pandoc}}

\newcommand{\anon}{1}

\newtheorem{assumption}{Assumption}[section]
\newcommand{\wt}{\widetilde}
\def\1{\mbox{\bf 1}}
\def\R{\mathbb{R}}

\def\P{\mathbb{P}}
\def\E{\mathbb{E}}

\def\R{\mathbb{R}}

\def\Z{\mathbb{Z}}

\def\dim{\mbox{dim}}

\newcolumntype{E}{>{\centering\arraybackslash}m{1.9cm}} 

\newcommand{\op}{\operatorname{op}}

\newtheorem{theo}{Theorem}
\newtheorem{lem}{Lemma}

\newtheorem{cor}{Corollary}
\newtheorem{Def/Prop}{Definition-Proposition}

\providecommand{\wt}{\widetilde}
\providecommand{\bV}{\mathbf{V}}
\providecommand{\sfE}{\mathsf{E}}
\providecommand{\sfA}{\mathsf{A}}
\providecommand{\sfB}{\mathsf{B}}

\begin{document}

\def\spacingset#1{\renewcommand{\baselinestretch}%
{#1}\small\normalsize} 
\spacingset{1}


\if1\anon
{
  \title{\bf  Uniform Gaussian Approximation for The  Quasi-Likelihood  Estimator for a Weakly Dependent Nonlinear Time Series Models}
  \author{ Zinsou Max Debaly  \\ Université du Québec à Montréal, Département de Mathématiques \\  Arsene Brice Zotsa-Ngoufack   \\ 
   Vanderbilt University, Department of Mathematics
    }
      \maketitle
 } \fi

\if0\anon
{
  \bigskip
  \bigskip
  \bigskip
  \begin{center}
    {\LARGE\bf Uniform Gaussian Approximation for The  Quasi-Likelihood  Estimator for a Weakly Dependent Nonlinear Time Series Models}
\end{center}
  
} \fi

 \begin{abstract}
We study estimation and inference for a semiparametric class  of 
time series models that specify only the conditional expectation, which is a known
link function applied to a linear combination of past observations and covariates.
The class covers count, binary, bounded and conditionally heteroskedastic
responses within a single formulation, and the parameter is estimated by a
quasi-likelihood estimating equation based on the first conditional moment. Under
stationarity and a weak-dependence condition expressed through the functional
dependence measure, we establish two results. First, using a Fuk--Nagaev
inequality for weakly dependent sequences, we show that the estimator is localized
in a shrinking neighbourhood of the true value with probability $1-o(n^{-1/2})$.
Second, combining a Berry--Esseen bound for weakly dependent sequences with a
Gaussian anti-concentration argument to control the remainder of the linear
expansion, we obtain a Berry--Esseen bound for linear projections of the
estimator, uniform over projection directions. From the projected bound we derive
studentized confidence intervals with explicit coverage error and a conservative
Bonferroni test for linear hypotheses on the parameters. For real data analysis, we extend the Beta
autoregression for double-bounded data to an arbitrary link given by the inverse
of a distribution function, and apply it to ten pairwise realized correlations of
large-cap technology-stock returns, using Nasdaq and Dow~Jones index returns as
covariates.
\end{abstract}

\noindent%
\vfill

\spacingset{1.8}

\section{Introduction}

Non-linear time series models have attracted considerable attention in recent
years. For integer-valued data, the leading examples are the Poisson
autoregression and INGARCH class \citep{fokianos2009poisson}, its log-linear
variant \citep{fokianos2012nonlinear}, the softplus specification
\citep{weibSoftplus}, and more general nonlinear count models
\citep{davis2016theory,ahmad2016poisson}, with heavy-tailed
\citep{gorgi2020beta} and multivariate \citep{fokianos2020multivariate}
extensions. For binary and categorical data, observation-driven models with
feedback are studied in \citet{moysiadis2014binary}; for double-bounded
continuous responses, such as rates, proportions and realized correlations, in
the beta autoregression of \citet{gorgi2023beta}; and for conditionally
heteroskedastic real-valued data in the GARCH class \citep{francq2019garch}.
Conditions for stationarity, ergodicity and weak dependence of the underlying
recursions are available in \citet{debaly_truquet_2021}. A common structure runs
through these models: the conditional mean is a known link function applied to a
linear combination of past observations and covariates. Exploiting only this
common structure, we study the semiparametric class
\begin{equation}\label{eq:intro-model}
  \E[Y_t\mid\mathcal F_{t-1}]
  =g\Big(\omega+\sum_{i=1}^{p}\alpha_i Y_{t-i}+\gamma^\top X_{t-1}\Big),
\end{equation}
where $(Y_t,X_t)_{t\in\Z}$ is stationary with natural filtration $\mathcal F_t$
and $g$ is a known link, and where the conditional distribution is left
unspecified. The parameter is estimated by the quasi-likelihood estimating
equation based on the first conditional moment, which reduces to conditional
least squares when $g$ is the identity and applies verbatim across count,
binary, bounded and conditionally heteroskedastic responses.

For this class, estimation theory is well developed, but the distributional
results available for the estimators are, in essence, statements of asymptotic
normality: after centring and scaling, the estimator converges in distribution to
a Gaussian law \citep{ahmad2016poisson,davis2016theory,debaly2023multivariate}.
Convergence in distribution is a statement about the limit only. It does not
quantify the accuracy of the Gaussian approximation at a given sample size, and
hence it does not, by itself, control the error of the confidence intervals and
tests built on it: the nominal level of a Wald interval is justified only in the
limit, with no indication of how large $n$ must be, nor of how the error depends
on the strength of the serial dependence. A Berry--Esseen bound replaces this
qualitative statement by a uniform one with an explicit rate, and thereby turns
asymptotic validity into a quantified coverage error.

The purpose of this paper is to establish such a bound, in a form that is
directly usable for inference, for the whole class \eqref{eq:intro-model}. To the
best of our knowledge, no finite-sample distributional result of this type is
available for   non-linear time series models, including in the
leading Poisson autoregression case; the existing literature stops at asymptotic
normality. Two features make the problem non-standard. First, the estimator is
defined implicitly, as the root of an estimating equation rather than as a sample
mean, so a Berry--Esseen theorem for sums does not apply directly: the
approximation error of the estimator is governed jointly by the leading score
term and by the second-order remainder of the linear expansion, and the two must
be controlled at compatible accuracies. Second, the estimating equation is a
quasi-likelihood, not a likelihood; the conditional law is unspecified, so the
expansions of the likelihood that underlie the classical refinements are not
available, and the standard error must be obtained from a sandwich matrix whose
fluctuation has itself to be controlled at the required rate.

Our contributions are the following.

\begin{enumerate}
\item[(i)] \emph{Localization with exceptional probability $o(n^{-1/2})$.} We
show that, with probability $1-o(n^{-1/2})$, the estimating equation admits a
unique root and that root lies within $O(\sqrt{\log n/n})$ of the true value
(Theorem~\ref{th:concentration}). What matters here is not the radius, which is
larger than the estimation rate, but the accuracy of the exceptional
probability: a distributional bound of order $n^{-1/2}$ cannot be transferred
onto an event whose complement is only $o(1)$, so the usual consistency
statements are not sufficient for what follows. The bound is obtained by a
Brouwer fixed-point construction on an event built from a Fuk--Nagaev
concentration inequality for functionally dependent sequences
\citep{liu2013probability}, applied to the score, to the Hessian and to the
second-order remainder simultaneously.

\item[(ii)] \emph{A Berry--Esseen bound for the studentized estimator, uniform
over directions.} We prove that the studentized statistic
satisfies a Berry--Esseen bound of order $\log n/\sqrt n$, uniformly in the
projection direction   and in the argument of the distribution function
(Theorem~\ref{th:BEprojection}). The Gaussian approximation of the leading score term is
supplied by the Berry--Esseen theorem of \citet{jirak2016berry} under the
functional dependence measure of \citet{Wu}; the contribution is the transfer of
that bound to the estimator, which we carry out by combining the localization of
(i) with a Gaussian anti-concentration argument applied to half-spaces. The bound
is stated for the studentized statistic, and hence is directly usable: this
requires controlling the empirical sandwich matrix, whose summands are of degree
four in the observations at the same accuracy.

\item[(iii)] \emph{Inference with explicit error.} From the projected bound we
obtain one- and two-sided confidence intervals for each coordinate with coverage
$1-\alpha+O(\log n/\sqrt n)$, and a Bonferroni omnibus test of a linear
hypothesis whose level exceeds the nominal one by at most $O(\log n/\sqrt n)$
(Corollary~\ref{cor:bonferroni-test}). Both statements are non-asymptotic in the
error term; neither is available from a central limit theorem alone.
\end{enumerate}

Beyond the coverage statements above, a Berry--Esseen bound is the first term of
the Edgeworth-type expansions on which the accuracy of resampling schemes rests.
In the independent case, \citet{barbe1995weighted} use such expansions to select
the distribution of the random weights in the weighted bootstrap, and the
perturbation bootstrap, which inserts random weights directly into the objective
function, has been shown to be second-order correct; \citet{das2025pebble}
establish this for the maximum likelihood estimator in logistic regression, where
the lattice structure of the binary response yields an error of order
$O(\log n/\sqrt n)$ that the perturbation bootstrap reduces to $o_p(n^{-1/2})$.
The bound obtained here is the corresponding leading term for the class
\eqref{eq:intro-model} under weak dependence, and is thus the natural starting
point for a second-order theory of resampling in these models. We do not pursue
that theory in this paper.

The rest of the paper is organized as follows. Section~2 introduces the model and
the quasi-likelihood estimating equation, and discusses the linear, count, binary
and bounded specifications that it covers. Section~3 contains the large-sample
theory: the localization result, the Berry--Esseen bound for the studentized
estimator, and the resulting confidence intervals and Bonferroni test.
Section~4 reports two Monte Carlo experiments, one on the localization radius and
one on the accuracy of the Gaussian approximation over a grid of projection
directions. Section~5 presents the application to realized correlations, and
Section~6 concludes. Proofs are collected in the supplementary material.
\section{Models and estimators}
  In this section, we introduce the proposed model and discuss several important special cases.

We recall that the  process $(Y_t, X_t)_{t\in\Z}$ satisfies the non-linear time
series model
\begin{equation}
  \label{eq::nonlinear}
  \E\!\left[Y_t \mid \mathcal{F}_{t-1}\right]
  = g\!\left(\omega + \sum_{i=1}^{p}\alpha_i\,Y_{t-i}
            + \gamma^{\top} X_{t-1}\right),
\end{equation}
where $\mathcal{F}_{t} = \sigma\!\left(Y_s, X_s : s \le t\right)$ is the natural
filtration and $g$ is a known link function. Let
$\theta = (\omega, \alpha_1, \ldots, \alpha_p, \gamma^{\top})^{\top}$ denote the
model parameter, and let $\theta_0$ be the true parameter value under which the
observations $(Y_t, X_t)_{t=-p+1,\ldots,n}$ are generated. For convenience we
write
\(
  \eta_t(\theta) = \omega + \sum_{i=1}^{p}\alpha_i\,Y_{t-i}
                 + \gamma^{\top} X_{t-1},
\)
so that \eqref{eq::nonlinear} reads
$\E[Y_t\mid\mathcal F_{t-1}] = g(\eta_t(\theta_0))$.

We estimate $\theta_0$ by the quasi-likelihood estimator $\widehat{\theta}_n$,
defined as a solution of the estimating equations
\begin{equation}
  \label{eq:qmle}
  \frac{1}{n}\sum_{t=1}^{n}
  \Big(Y_t - g\big(\eta_t(\widehat{\theta}_n)\big)\Big)\,Z_{t-1} = 0,
  \qquad
  Z_{t-1} = \big(1, Y_{t-1}, \ldots, Y_{t-p}, X_{t-1}^{\top}\big)^{\top}.
\end{equation}

 Some examples are given as follows.
\subsection{Linear models and the least squares estimator}

The linear models are those for which the link $g$ in \eqref{eq::nonlinear} is
the identity. The conditional mean is then an affine function of the regressor
vector $Z_{t-1}=(1,Y_{t-1},\dots,Y_{t-p},X_{t-1}^{\top})^{\top}$:
\[
  \eta_t(\theta)=\theta^{\top}Z_{t-1},
  \qquad
  \E[Y_t\mid\mathcal F_{t-1}]=\theta_0^{\top}Z_{t-1}.
\]
Three standard models belong to this class. In each case we write the
data-generating equation and verify that its conditional mean has the form
\eqref{eq::nonlinear} with $g=\mathrm{id}$.

\emph{(a) Autoregressive model $\mathrm{AR}(p)$}. The
process $(Y_t)$ is real-valued and satisfies
\begin{equation}
\label{eq:ar}
  Y_t=\omega+\sum_{i=1}^{p}\alpha_i\,Y_{t-i}+\gamma^{\top}X_{t-1}+\varepsilon_t,
  \qquad \E[\varepsilon_t\mid\mathcal F_{t-1}]=0,
\end{equation}
where the innovation $(\varepsilon_t)$ is a martingale-difference sequence.
Taking conditional expectations gives
$\E[Y_t\mid\mathcal F_{t-1}]=\eta_t(\theta_0)$.

\emph{(b) Poisson model $\mathrm{INARCH}(p)$}. The process $(Y_t)$ is $\Z_+$-valued and, given the
past, is Poisson distributed with a linear intensity:
\begin{equation}
\label{eq:inarch}
  Y_t\mid\mathcal F_{t-1}\sim\mathrm{Poisson}(\lambda_t),
  \qquad
  \lambda_t=\omega+\sum_{i=1}^{p}\alpha_i\,Y_{t-i}+\gamma^{\top}X_{t-1},
\end{equation}
with $\omega>0$, $\alpha_i\ge 0$ and $X_{t-1}$ such that $\lambda_t>0$ almost
surely. Then $\E[Y_t\mid\mathcal F_{t-1}]=\lambda_t=\eta_t(\theta_0)$.

\emph{(c) Model $\mathrm{ARCH}(p)$}. Here the linear
structure appears after squaring. Let the underlying process $(Y_t^{*})$ satisfy
\begin{equation}
\label{eq:arch}
  Y_t^{*}=\sigma_t\,\varepsilon_t,
  \qquad
  \sigma_t^{2}=\omega+\sum_{i=1}^{p}\alpha_i\,(Y_{t-i}^{*})^{2}
              +\gamma^{\top}X_{t-1},
\end{equation}
where $(\varepsilon_t)$ is i.i.d., independent of $\mathcal F_{t-1}$, with
$\E\varepsilon_t=0$ and $\E\varepsilon_t^{2}=1$, and $\omega>0$,
$\alpha_i\ge 0$. Define the observed process as the squared returns,
\(
  Y_t:=(Y_t^{*})^{2}.
\)
Since $Y_{t-i}=(Y_{t-i}^{*})^{2}$ and
$\E[\varepsilon_t^{2}\mid\mathcal F_{t-1}]=1$,
\[
  \E[Y_t\mid\mathcal F_{t-1}]=\sigma_t^{2}
  =\omega+\sum_{i=1}^{p}\alpha_i\,Y_{t-i}+\gamma^{\top}X_{t-1}
  =\eta_t(\theta_0),
\]
so $(Y_t)$ obeys \eqref{eq::nonlinear} with $g=\mathrm{id}$.  

In all three models the error $e_t:=Y_t-\eta_t(\theta_0)$ is a
martingale difference, $\E[e_t\mid\mathcal F_{t-1}]=0$. They differ in the conditional law of $e_t$: additive noise $\varepsilon_t$ for
$\mathrm{AR}(p)$, the centred count $Y_t-\lambda_t$  for $\mathrm{INARCH}(p)$, and the multiplicative form
$\sigma_t^{2}(\varepsilon_t^{2}-1)$ for the squared $\mathrm{ARCH}(p)$; in particular the conditional variance is constant for $\mathrm{AR}(p)$ but state-dependent for the other two.

The estimating equations \eqref{eq:qmle} reduce to the linear
normal equations
\[
  \frac1n\sum_{t=1}^{n}\big(Y_t-\widehat\theta_n^{\top}Z_{t-1}\big)Z_{t-1}=0
  \iff
  \Big(\frac1n\sum_{t=1}^{n}Z_{t-1}Z_{t-1}^{\top}\Big)\widehat \theta_n
  =\frac1n\sum_{t=1}^{n}Y_t\,Z_{t-1}.
\]
Whenever the     matrix
$\frac1n\sum_{t=1}^{n}Z_{t-1}Z_{t-1}^{\top}$ is invertible, this has the
closed-form solution
\begin{equation}
\label{eq:lse}
  \widehat\theta_n
  =\Big(\frac1n\sum_{t=1}^{n}Z_{t-1}Z_{t-1}^{\top}\Big)^{-1}
   \Big(\frac1n\sum_{t=1}^{n}Y_t\,Z_{t-1}\Big),
\end{equation}
the (conditional) least squares estimator, i.e.\ the minimizer of
$\theta\mapsto\frac1{2n}\sum_{t=1}^{n}(Y_t-\theta^{\top}Z_{t-1})^{2}$.
\subsection{Non-linear models for count autoregressions}

Linear count models are subject to two structural limitations. To keep the
conditional mean $\lambda_t=\eta_t(\theta_0)$ positive, the parameters must be
sign-constrained ($\omega>0$, $\alpha_i\ge 0$, and $\gamma$ restricted so that
$\lambda_t>0$ almost surely); as a consequence, the model can only reproduce
non-negative autocorrelations. Non-linear links have been introduced to lift
both restrictions.

\medskip
\emph{(a) Log-linear model}. This model is introduced by \citet{fokianos2012nonlinear}. Writing
$\nu_t=\log\lambda_t$, the log-linear Poisson autoregression specifies
\begin{equation}
\label{eq:loglinear}
  Y_t\mid\mathcal F_{t-1}\sim\mathrm{Poisson}(\lambda_t),
  \qquad
  \nu_t=\log\lambda_t
       =\omega +\sum_{i=1}^{p}a_i\,\log(1+Y_{t-i}) 
        +\gamma^{\top}X_{t-1}.
\end{equation}
The logarithmic link guarantees $\lambda_t>0$ without sign constraints and
allows negative dependence. However,  the
parameters act on the log-scale and on log-transformed lags, which makes them
harder to interpret, and the attainable range of autocorrelations is
restricted.

\emph{(b) Softplus Poisson model}. \citet{weibSoftplus} proposed the Softplus link function defined  with scale $c>0$ by
\(
  s_c(x)=c\,\log\!\big(1+e^{x/c}\big),\, x\in\R .
\)
It is smooth, strictly increasing, and strictly positive, with derivative equal
to the logistic function,
$s_c'(x)=\big(1+e^{-x/c}\big)^{-1}\in(0,1)$, and second derivative
$s_c''(x)=c^{-1}s_c'(x)\big(1-s_c'(x)\big)$.   The softplus
Poisson autoregression is
\begin{equation}
\label{eq:softplus-poisson}
  Y_t\mid\mathcal F_{t-1}\sim\mathrm{Poisson}(\lambda_t),
  \qquad
  \lambda_t=s_c\!\Big(\omega+\sum_{i=1}^{p}\alpha_i\,Y_{t-i}
                     +\gamma^{\top}X_{t-1}\Big).
\end{equation}
Because $s_c>0$, no sign constraint is imposed on
$(\omega,\alpha_1,\dots,\alpha_p,\gamma)$, so the model accommodates negative autocorrelation. Since
$s_c(x)\approx x$ once $\lambda_t$ is not too small, the model is then close to
the linear INARCH$(p)$, and the parameters retain the interpretation they have
in the linear case. Model
\eqref{eq:softplus-poisson} is exactly \eqref{eq::nonlinear} with the known
link $g=s_c$ 
and the estimating equations \eqref{eq:qmle} apply verbatim with $g=s_c$.  

\medskip
\emph{(c) Extension to the Negative-binomial softplus model}. The Poisson
specification imposes equidispersion, $\mathrm{Var}(Y_t\mid\mathcal
F_{t-1})=\lambda_t$, which is often violated. Retaining the softplus link but
replacing the conditional law by a negative binomial one accommodates
overdispersion while remaining within \eqref{eq::nonlinear}. Let $r>0$ be a
fixed dispersion parameter and, given the past, let $Y_t$ follow a negative
binomial law $\mathrm{NB}(r,p_t)$ counting failures before the $r$-th success,
\[
  \P\big(Y_t=y\mid\mathcal F_{t-1}\big)
  =\binom{y+r-1}{y}\,p_t^{\,r}(1-p_t)^{y},\qquad y=0,1,2,\dots,
\]
whose conditional mean and variance are
$\mu_t=r(1-p_t)/p_t$ and $r(1-p_t)/p_t^{2}=\mu_t+\mu_t^{2}/r$. We model the mean
through the softplus link,
\begin{equation}
\label{eq:softplus-nb}
  \mu_t=\E[Y_t\mid\mathcal F_{t-1}]
       =s_c\!\Big(\omega+\sum_{i=1}^{p}\alpha_i\,Y_{t-i}
                 +\gamma^{\top}X_{t-1}\Big)
       =s_c\big(\eta_t(\theta_0)\big).
\end{equation}
Thus \eqref{eq:softplus-nb} is again \eqref{eq::nonlinear} with the same link
$g=s_c$ and the same regressor $\eta_t(\theta)=\theta^{\top}Z_{t-1}$; the
conditional mean, and therefore the estimating equations \eqref{eq:qmle}, are
identical to the Poisson case \eqref{eq:softplus-poisson}. The two models differ
only through the conditional variance,
\[
  \mathrm{Var}(Y_t\mid\mathcal F_{t-1})
  =s_c\big(\eta_t(\theta_0)\big)+\frac{s_c\big(\eta_t(\theta_0)\big)^{2}}{r}
  \;>\;\mu_t ,
\]
which is overdispersed and recovers the Poisson variance as $r\to\infty$.
Consequently the estimator $\widehat\theta_n$ of \eqref{eq:qmle} is valid for
both models: being a first-moment (quasi-likelihood) estimating equation, it
depends on the conditional law only through $\mu_t=s_c(\eta_t)$, and the change
from Poisson to negative binomial affects the efficiency of $\widehat\theta_n$
but neither its definition.

\subsection{Autoregressions for binary data}

For binary responses $Y_t\in\{0,1\}$, the conditional mean is the conditional
success probability, $\E[Y_t\mid\mathcal F_{t-1}]=\P(Y_t=1\mid\mathcal
F_{t-1})=\pi_t$. A unified framework is obtained by taking the link $g$ in
\eqref{eq::nonlinear} to be a known cumulative distribution function
$F:\R\to(0,1)$ (the response function), so that
\begin{equation}
\label{eq:binary}
  Y_t\mid\mathcal F_{t-1}\sim\mathrm{Bernoulli}(\pi_t),
  \qquad
  \pi_t=F\!\Big(\omega+\sum_{i=1}^{p}\alpha_i\,Y_{t-i}
               +\gamma^{\top}X_{t-1}\Big)=F\big(\eta_t(\theta_0)\big).
\end{equation}
This is the dynamic binary-response model with a general link of
\citet{kauppi}; its observation-driven form with feedback is
studied in \citet{moysiadis2014binary}, and the mixing and asymptotic theory of the maximum likelihood estimator
for the case in which lagged responses enter as regressors is established in \citet{deJong2011}.  
The estimating equations \eqref{eq:qmle} apply with $g=F$ and  it does not coincide with the Bernoulli maximum
likelihood estimator, except the case of the logistic model.
 
\subsection{Autoregressions for bounded \texorpdfstring{$(0,1)$}{(0,1)} data}
For bounded responses $Y_t\in(0,1)$, such as rates and proportions, the
conditional mean $\mu_t=\E[Y_t\mid\mathcal F_{t-1}]\in(0,1)$ is again modelled
through a known cumulative distribution function $F:\R\to(0,1)$, and the
conditional law is taken to be a Beta distribution. In the mean--precision
parametrization,
\begin{equation}
\label{eq:beta}
  Y_t\mid\mathcal F_{t-1}\sim
  \mathrm{Beta}\big(\mu_t\phi,\,(1-\mu_t)\phi\big),
  \qquad
  \mu_t=F\!\Big(\omega+\sum_{i=1}^{p}\alpha_i\,Y_{t-i}
               +\gamma^{\top}X_{t-1}\Big)=F\big(\eta_t(\theta_0)\big),
\end{equation}
where $\phi>0$ is a precision parameter; the two shape parameters
$\mu_t\phi$ and $(1-\mu_t)\phi$ give conditional mean $\mu_t$ and conditional
variance $\mu_t(1-\mu_t)/(1+\phi)$. As in the binary case, $F$ is a strictly
increasing, continuously differentiable distribution function, the logit
 and the probit   being the standard choices. This is the beta autoregressive model  built on the beta regression of
\citet{ferrari2004beta}. The estimating equations
\eqref{eq:qmle} apply with $g=F$ and do not coincide with the Beta maximum
likelihood estimator.
\section{Large sample properties}
In this section, we establish a concentration result (Theorem~\ref{th:concentration}), which is then used to derive a Berry--Esseen-type bound (Theorem~\ref{th:BEprojection}). We also propose a Bonferroni-type testing procedure for hypotheses on the model parameters (Corollary~\ref{cor:bonferroni-test}).

Starting from \eqref{eq:qmle} and a first-order Taylor expansion of
$g(\eta_t(\cdot))$ around $\theta_0$, we obtain
\begin{equation}
  \label{eq:expansion}
  0 = \frac{1}{n}\sum_{t=1}^{n}\Big(Y_t - g\big(\eta_t(\widehat\theta_n)\big)\Big)Z_{t-1}
    = S_n - A_n\,\widehat h_n - R_n(\widehat h_n),
\end{equation}
where $\widehat h_n = \widehat\theta_n - \theta_0$ and
\[
  S_n = \frac{1}{n}\sum_{t=1}^{n}\Big(Y_t - g\big(\eta_t(\theta_0)\big)\Big)Z_{t-1},
  \qquad
  A_n = \frac{1}{n}\sum_{t=1}^{n} g'\big(\eta_t(\theta_0)\big)\,Z_{t-1}Z_{t-1}^{\top},
\]
and the second-order remainder is
\[
  R_n(h) = \frac{1}{n}\sum_{t=1}^{n}
    \left(\int_0^1 \Big[g'\big(\eta_t(\theta_0 + s\,h)\big)
                       - g'\big(\eta_t(\theta_0)\big)\Big]\,\mathrm{d}s\right)
    Z_{t-1}Z_{t-1}^{\top}\,h .
\]
Throughout, $d=1+p+\dim(X_{t-1})$ denotes the dimension of $\theta$, so that
$\theta,Z_{t-1}\in\R^{d}$.

For the linear models, $g$ is the identity, hence $g'\equiv 1$ is constant and
the integrand in $R_n$ vanishes, so $R_n\equiv 0$ and
$A_n=\frac1n\sum_{t=1}^{n}Z_{t-1}Z_{t-1}^{\top}$; the expansion reduces to
$\widehat h_n=A_n^{-1}S_n$. In general, on the event where $A_n$ is invertible,
equation \eqref{eq:expansion} is equivalent to the fixed-point equation
$\widehat h_n=\Psi_n(\widehat h_n)$, where
\(
  \Psi_n(h)=A_n^{-1}\big\{S_n-R_n(h)\big\}.
\)

Concentration of the estimator is obtained by exhibiting a high-probability
event $\mathcal E_n$, depending on a radius sequence $r_n$, on which $\Psi_n$ is
well defined and maps the ball $\overline B(0,r_n)=\{h\in\R^{d}:\|h\|\le r_n\}$
into itself, that is, $\sup_{\|h\|\le r_n}\|\Psi_n(h)\|\le r_n$. On
$\mathcal E_n$, Brouwer's fixed-point theorem then yields a solution
$\widehat h_n\in\overline B(0,r_n)$, giving the concentration
$\|\widehat\theta_n-\theta_0\|\le r_n$. It therefore suffices to bound $S_n$,
$A_n$ and $\sup_{\|h\|\le r_n}\|R_n(h)\|$ on events of high probability, after
which the result follows from a union bound. To construct these events we use
the Fuk--Nagaev inequality for dependent sequences with summable functional
dependence coefficients, established by \citet{liu2013probability} under the
functional dependence measure of \citet{Wu}.

We recall the functional dependence measure of \citet{Wu} for causal time
series. Assume that a time series $(U_t)_{t\in\Z}$ admits a causal
(Bernoulli-shift) representation
$U_t=f(\varepsilon_t,\varepsilon_{t-1},\ldots)$, $t\in\Z$, for a measurable
function $f$ and an \emph{i.i.d.} sequence $(\varepsilon_t)_{t\in\Z}$. Let
$\widetilde\varepsilon_0$ be an independent copy of $\varepsilon_0$, and define
the coupled sequence $(\widetilde\varepsilon_t)_{t\in\Z}$ by
$\widetilde\varepsilon_t=\varepsilon_t$ for $t\neq0$ and $\widetilde\varepsilon_0$
in place of $\varepsilon_0$, together with
$\widetilde U_t=f(\widetilde\varepsilon_t,\widetilde\varepsilon_{t-1},\ldots)$.
For any order $q\ge1$ and $t\ge0$, the functional dependence coefficients of
$(U_t)_{t\in\Z}$ are
\[
  \theta_{q,t}=\big\|U_t-\widetilde U_t\big\|_q = \E^{1/q} \big\|U_t-\widetilde U_t\big\|^q,
  \qquad
  \Theta_{m,q}=\sum_{i=m}^{\infty}\theta_{q,i},
\]
and $(U_t)_{t\in\Z}$ is $q$-stable if $\Theta_{0,q}<\infty$. For non-linear time
series with exogenous covariates defined through stochastic recursive equations,
\citet{debaly_truquet_2021} give sufficient conditions under which a $q$-stable
solution exists when the covariates are themselves $q$-stable (Section~3
therein), while \citet{liu2013probability} establish Fuk--Nagaev concentration
bounds for $q$-stable processes.

Since $S_n$, $A_n$ and $R_n(h)$ are empirical means of functions of the sequence
$(X_t,Y_t)_{t\in\Z}$, we impose a stability assumption on $(X_t,Y_t)_{t\in\Z}$
and regularity assumptions on $g$ and $g'$ ensuring that these empirical means
are averages of $q$-stable sequences.

\begin{assumption}[Stationarity and dependence]
\label{ass:stationaryDependence}
The process $(Y_t,X_t)_{t\in\Z}$ is stationary and ergodic. There exist
$q\ge 3$, $\delta\in(0,2q]$, $\rho\in(0,1)$ and $C<\infty$ such that
\[
  \Theta_{m,q+\delta}(Y)\le C\rho^{m},
  \qquad
  \Theta_{m,q+\delta}(X)\le C\rho^{m},
  \qquad m\ge 0,
\]
and $\E|Y_0|^{\,2q(q+\delta)/\delta}<\infty$,
$\E\|X_0\|^{\,2q(q+\delta)/\delta}<\infty$.
\end{assumption}

 Establishing concentration for $S_n$, $A_n$ and $R_n(h)$ requires the products
they average to be $q$-stable.   More
precisely, the fluctuation of such a product factors, through H\"older's
inequality with $1/q=1/(q+\delta)+1/r$, into an increment of the underlying
process $(Y_t,X_t)$ measured at order $q+\delta$ and an envelope of degree two
measured in $L^{r}$. Order-$q$ stability of the product is thus recovered by
requiring $(Y_t,X_t)$ to be stable at the strictly higher order $q+\delta$, and
the degree-two envelope is controlled by a finite moment of order
$2r=2q(q+\delta)/\delta$. These are the two conditions of
Assumption~\ref{ass:stationaryDependence}. The margin $\delta>0$ is the excess
stability order that forming products costs; it may be taken arbitrarily small,
at the price of a larger moment $2q(q+\delta)/\delta$, which diverges as
$\delta\downarrow0$. The two conditions are genuinely distinct in the regime $\delta<2q$: there,
geometric decay of $\Theta_{m,q+\delta}$ yields only a finite moment of order
$q+\delta$, strictly smaller than the required order $2q(q+\delta)/\delta$
(indeed $2q(q+\delta)/\delta>q+\delta$ if and only if $\delta<2q$). This is the
relevant regime, since one takes $\delta$ small, the required order
$2q(q+\delta)/\delta$ diverging as $\delta\downarrow0$.
 The lower bound $q\ge3$ enters at the
concentration step: it guarantees that the exceptional event on which the
fixed-point construction of $\widehat h_n$ fails has probability $o(n^{-1/2})$,
which is what the union bound over the controls of $S_n$, $A_n$ and $R_n$
requires.

A broad class of models defined through stochastic recursive equations fits
\eqref{eq::nonlinear} and satisfies Assumption~\ref{ass:stationaryDependence}
under a local Lipschitz condition on the recursion together with geometric decay
and the moment condition on the covariate process; see conditions
\textbf{A3$'$} or \textbf{B4} in \citet{debaly_truquet_2021}, which also give
conditions for stationarity, ergodicity and finite moments. Geometric decay is
easy to verify in practice: for a Lipschitz iterated-random-function recursion
the functional dependence coefficients satisfy $\theta_{q,t}\le 2K_qL_q^{t}$
with $L_q<1$, whence $\Theta_{m,q}=O(L_q^{m})$; see Example~1 in
\citet{liu2013probability} and Proposition~2 in \citet{debaly_truquet_2021}.
\begin{assumption}
\label{ass:hypo_sur_g}
The link $g:\R\to\R$ is of class $\mathcal C^1$, and there exists $C>0$ such
that, for all $u,v\in\R$,
\[
  |g(u)-g(v)|\le C\,|u-v|
  \qquad\text{and}\qquad
  |g'(u)-g'(v)|\le C\,|u-v| .
\]
\end{assumption}

The first condition states that $g$ is Lipschitz (equivalently, $g'$ is
bounded), and the second that $g'$ is Lipschitz; together they control $A_n$ and
the remainder $R_n$ in \eqref{eq:expansion}. For the linear models both hold
trivially: $g=\mathrm{id}$ gives $|g(u)-g(v)|=|u-v|$ and $g'\equiv1$, so the
second inequality holds with any $C\ge0$. For the count autoregressions the link
is the softplus $s_c$, whose derivative is the logistic function
($s_c'\in(0,1)$, $\|s_c''\|_\infty\le 1/(4c)$), and for the binary and bounded
models the link is a distribution function   with bounded
Lipschitz density; in all these cases both conditions of
Assumption~\ref{ass:hypo_sur_g} are satisfied.

\begin{lem}\label{lem:controlScore}
    Under the conditions of Assumption \ref{ass:stationaryDependence} and \ref{ass:hypo_sur_g}, there is $a_S>0,$ such that 
    $$
    \P\Big(\|S_n\|\geq a_S\sqrt{\frac{\log n}{n}}\Big) = o(n^{-1/2}).
    $$
\end{lem}

We now turn to the Hessian term $A_n$. Under
Assumption~\ref{ass:stationaryDependence} the process $(Y_t,X_t)$ is stationary
and ergodic with a finite moment of order $2q(q+\delta)/\delta$;   the ergodic theorem gives
\(
  A_n\xrightarrow[n\to\infty]{\text{a.s.}}
  A_0:=\E\big[g'(\eta_1(\theta_0))\,Z_0Z_0^{\top}\big],
  \, \lambda_0:=\lambda_{\min}(A_0).
\)
Assumption~\ref{ass:pdHessian} below, together with the strict positivity of
$g'$, ensures $\lambda_0>0$. Applying the Fuk--Nagaev inequality for
functionally dependent sequences to the entries of $A_n$, we show that
$\lambda_{\min}(A_n)$ remains bounded away from $0$ with high probability.

\begin{assumption}
\label{ass:pdHessian}
For every $u\in\R^{d}$, $\ \P(u^{\top}Z_0=0)=1\ \Rightarrow\ u=0$. In addition, assume $g'>0$ .
\end{assumption}

\begin{lem}
\label{lem:pbHessianBound}
Under Assumptions~\ref{ass:stationaryDependence}, \ref{ass:hypo_sur_g} and
\ref{ass:pdHessian},  
\[
  \P\Big(\lambda_{\min}(A_n)>\lambda_0/2\Big)=1-o(n^{-1/2}).
\]
\end{lem}

For the remainder term, the Lipschitz property of $g'$ yields a deterministic
bound. Writing
\(
  \bar R_n:=\frac1n\sum_{t=1}^{n}
    \Big(1+\sum_{i=1}^{p}|Y_{t-i}|+\|X_{t-1}\|\Big)\|Z_{t-1}\|^{2},
\)
one can note that  that
\begin{equation}\label{eq:Rn-det}
  \sup_{\|h\|\le r_n}\|R_n(h)\|\;\le\;\frac{C}{2}\,r_n^{2}\,\bar R_n,
\end{equation}
where $C$ is the Lipschitz constant of $g'$ from
Assumption~\ref{ass:hypo_sur_g}. Under Assumption~\ref{ass:stationaryDependence}
the process is stationary and ergodic with a finite moment of order
$2q(q+\delta)/\delta$, so
$\big(1+\sum_{i=1}^{p}|Y_{-i}|+\|X_{-1}\|\big)\|Z_{-1}\|^{2}$ is integrable and
the ergodic theorem gives
\(
  \bar R_n\xrightarrow[n\to\infty]{\text{a.s.}}
  \E\!\Big[\Big(1+\sum_{i=1}^{p}|Y_{-i}|+\|X_{-1}\|\Big)\|Z_{-1}\|^{2}\Big];
\)
the Fuk--Nagaev inequality then controls its fluctuation. Since the factor
$r_n^{2}$ is common to both sides of \eqref{eq:Rn-det}, the event to be
controlled does not depend on $r_n$. The result reads as follows.
\begin{lem}
\label{lem:controlRn}
Under Assumptions~\ref{ass:stationaryDependence} and \ref{ass:hypo_sur_g}, there exists $a_R>0$ such that
\[
  \P\Big(\sup_{\|h\|\le r_n}\|R_n(h)\|\ge a_R\,r_n^{2}\Big)=o(n^{-1/2}).
\]
\end{lem}

Lemmas~\ref{lem:controlScore}, \ref{lem:pbHessianBound}, and \ref{lem:controlRn} provide the high-probability events needed to control $S_n, A_n$, and $R_n(h)$, and together yield a concentration result for the estimator.
\begin{theo}
\label{th:concentration}
Consider the model \eqref{eq::nonlinear} and the estimator \eqref{eq:qmle}, and
suppose Assumptions~\ref{ass:stationaryDependence}, \ref{ass:hypo_sur_g} and
\ref{ass:pdHessian} hold. Then there exists $a>0$ such that, with
probability $1-o(n^{-1/2})$, the estimating equation \eqref{eq:qmle} admits a
unique solution $\widehat\theta_n$, and it satisfies
\[
  \|\widehat\theta_n-\theta_0\|\le a\sqrt{\tfrac{\log n}{n}} .
\]
\end{theo}

Theorem~\ref{th:concentration} should not be read as a statement about the convergence rate of $\widehat\theta_n$. In fixed dimension the estimator is root-$n$ consistent, a fact that is well established; the radius
$\sqrt{\log n/n}$ is not the estimation rate but the radius of a neighbourhood on
which the fixed-point construction of $\widehat\theta_n$ can be carried out. The
logarithmic factor is the price paid to obtain a clean bound on the probability
of the exceptional event, namely $1-o(n^{-1/2})$: enlarging the radius from $n^{-1/2}$ to $\sqrt{\log n/n}$ is what makes this probability decay faster than $n^{-1/2}$, which is in turn what the projected Berry--Esseen bound of
Theorem~\ref{th:BEprojection} requires.
 
  For regression with continuous, independent data, the factor can be removed under Cramér-type
smoothness conditions on the   errors, through
\citep{bhattacharya1978validity}. For discrete-valued
responses, however, no such conditional density is available, and the lattice
structure of the score obstructs these expansions; the $\log n$ factor then
appears difficult to discard even in the independent case, as observed for logistic regression \citep{das2025pebble}.

From \eqref{eq:expansion}, we have the expansion
\begin{eqnarray}
    \sqrt{n}\widehat L_n^{-1/2} \widehat{A}_n (\widehat \theta_n - \theta_0)  & = & \sqrt{n} L_0^{-1/2} S_n + \sqrt{n} \Big(L_n^{-1/2}-L_0^{-1/2} \Big)S_n +  L_n^{-1/2} R_n(h_n)
    \\
     && + \Big(\widehat L_n^{-1/2}\widehat A_n - L_n^{-1/2}A_n\Big)h_n \nonumber
\end{eqnarray}
where $L_0 = \E \Big[(Y_1 - g(\eta_1(\theta_0)))^2 Z_{0}Z_{0}^\top\Big], \, L_n = L_n(\theta_0) \text{ and } \widehat L_n = L_n(\widehat \theta_n)$ with
$$
\widehat A_n = \frac{1}{n}\sum_{t=1}^n g'(\eta_t(\widehat \theta_n))Z_{t-1} Z_{t-1}^\top, \, L_n(\theta) = \frac{1}{n}\sum_{t=1}^n (Y_t - g(\eta_t(\theta)))^2 Z_{t-1}Z_{t-1}^\top.
$$
The following assumption ensures that the minimum eigenvalue $\lambda_L$ of $L_0$ is bounded away from $0.$
\begin{assumption}
\label{ass:nondegError}
The conditional variance of the response is almost surely positive:
\[
  \E\big[(Y_1-g(\eta_1(\theta_0)))^{2}\mid\mathcal F_0\big]>0
  \qquad\text{a.s.}
\]
\end{assumption}

\begin{lem}\label{lem:be-projection}
Suppose Assumptions~\ref{ass:stationaryDependence},
\ref{ass:hypo_sur_g}, \ref{ass:pdHessian} and \ref{ass:nondegError} hold. Then
there exists a finite constant $B^\ast$, depending only on $\lambda_L$ and the
dependence constants $(\rho,C)$ of Assumption~\ref{ass:stationaryDependence},
such that
\[
  \sup_{\|u\|=1}\ \sup_{x\in\R}
  \Big|\P\big(u^\top\sqrt n\,L_0^{-1/2}S_n\le x\big)-\Phi(x)\Big|
  \le \frac{B^\ast}{\sqrt n}.
\]
\end{lem}

Lemma ~\ref{lem:be-projection} controls the leading term
$u^\top\sqrt n\,L_0^{-1/2}S_n$ of the expansion \eqref{eq:expansion}, uniformly in $\|u\|=1$. It remains to show that the other terms of that expansion are
negligible. Collecting them, we set, for $\|u\|=1$,
\[
  \widetilde R_n(u):=u^\top\!\Big(
     \sqrt n\big(L_n^{-1/2}-L_0^{-1/2}\big)S_n
     +L_n^{-1/2}R_n(h_n)
     +\big(\widehat L_n^{-1/2}\widehat A_n-L_n^{-1/2}A_n\big)h_n\Big),
  \qquad h_n=\widehat\theta_n-\theta_0,
\]
so that the studentized statistic
$u^\top\sqrt n\,\widehat L_n^{-1/2}\widehat A_n(\widehat\theta_n-\theta_0)$ equals
$u^\top\sqrt n\,L_0^{-1/2}S_n+\widetilde R_n(u)$. Our aim is to bound
$\sup_{\|u\|=1}|\widetilde R_n(u)|$ on an event of probability
$1-o(n^{-1/2})$, at the scale $\log n/\sqrt n$; combined with
Lemma ~\ref{lem:be-projection}, this transfers the Berry--Esseen bound
from the leading projection to the studentized statistic. Unlike the terms
$S_n$, $A_n$ and $R_n$ treated so far, the remainder involves the empirical
variance matrix $L_n$, whose summand $(Y_t-g(\eta_t(\theta_0)))^2Z_{t-1}Z_{t-1}^\top$
is of degree four in the observations. Controlling its fluctuation at the
required rate therefore calls for a moment of higher order than
Assumption~\ref{ass:stationaryDependence} supplies, which is the reason for the
strengthened moment condition stated in Assumption~\ref{ass:moment3r} below.

\begin{assumption}[Moment strengthening for studentization]\label{ass:moment3r}
The marginal moment condition of Assumption~\ref{ass:stationaryDependence} holds
at the order $3q(q+\delta)/\delta$: $\E|Y_0|^{3q(q+\delta)/\delta}<\infty$ and
$\E\|X_0\|^{3q(q+\delta)/\delta}<\infty$.
\end{assumption}

\begin{lem}\label{lem:be-remainder}
Suppose Assumptions~\ref{ass:stationaryDependence}, \ref{ass:hypo_sur_g},
\ref{ass:pdHessian}, \ref{ass:nondegError} and \ref{ass:moment3r} hold. 
Then there is $\widetilde a_R>0$ such that
\[
  \P\Big(\sup_{\|u\|=1}\big|\widetilde R_n(u)\big|
  \ge \widetilde a_R\,\frac{\log n}{\sqrt n}\Big)=o(n^{-1/2}),
\]
and in particular
$\sup_{\|u\|=1}\P\big(|\widetilde R_n(u)|\ge\widetilde a_R\log n/\sqrt n\big)=o(n^{-1/2})$.
\end{lem}

Having controlled the leading projection (Lemma ~\ref{lem:be-projection})
and the remainder (Lemma~\ref{lem:be-remainder}), we combine the two into a
Berry--Esseen bound for the studentized statistic
$ u^\top \sqrt n\,\widehat L_n^{-1/2}\widehat A_n(\widehat\theta_n-\theta_0)$.

\begin{theo}\label{th:BEprojection}
Consider the model \eqref{eq::nonlinear} and the estimator \eqref{eq:qmle}.
Suppose Assumptions~\ref{ass:stationaryDependence}, \ref{ass:hypo_sur_g},
\ref{ass:pdHessian}, \ref{ass:nondegError} and \ref{ass:moment3r} hold. Then
\[
  \sup_{\|u\|=1}\ \sup_{x\in\R}
  \Big|\P\big(u^\top\sqrt n\,\widehat L_n^{-1/2}\widehat A_n(\widehat\theta_n-\theta_0)\le x\big)-\Phi(x)\Big|
  =O\!\Big(\frac{\log n}{\sqrt n}\Big).
\]
\end{theo}

One can easily show  the following result.

\begin{lem}\label{lem:approxVar}
Under Assumptions~\ref{ass:stationaryDependence}, \ref{ass:hypo_sur_g},
\ref{ass:pdHessian}, \ref{ass:nondegError} and \ref{ass:moment3r}, there is $C<\infty$ such that
\[
  \P\Big(\big\|\widehat A_n^{-1}\widehat L_n\widehat A_n^{-1}
              -A_0^{-1}L_0A_0^{-1}\big\|_{\op}\le C\sqrt{\tfrac{\log n}{n}}\Big)
  =1-o(n^{-1/2}).
\]
\end{lem}

By Lemma~\ref{lem:approxVar}, the sandwich estimator
$ \widehat A_n^{-1}\widehat L_n\widehat A_n^{-1}$ is consistent for
$A_0^{-1}L_0A_0^{-1}$ at rate $\sqrt{\log n/n}$ on an event of probability
$1-o(n^{-1/2})$. Fix a coordinate $j\in\{1,\dots,d\}$ and let $u_j$ be the
$j$-th canonical basis vector. Then
$u_j^\top\sqrt n(\widehat\theta_n-\theta_0)=\sqrt n(\widehat\theta_{n,j}-\theta_{0,j})$,
and from the expansion \eqref{eq:expansion},
\[
  \frac{\sqrt n\,(\widehat\theta_{n,j}-\theta_{0,j})}{\widehat s_{n,j}}
  = w_{j,n}^\top\widehat H_n+o_{\P}\!\Big(\frac{\log n}{\sqrt n}\Big),
  \qquad
  \widehat s_{n,j}^2=(\widehat V_n)_{jj},\quad
  w_{j,n}=\frac{\widehat L_n^{1/2}\widehat A_n^{-1}u_j}{\widehat s_{n,j}},\ \|w_{j,n}\|=1,
\]
where $\widehat H_n=\sqrt n\,\widehat L_n^{-1/2}\widehat A_n(\widehat\theta_n-\theta_0)$.
The direction $w_{j,n}$ is data-dependent; however, on the event
$\mathcal G_n$ it lies within $O(\sqrt{\log n/n})$ of the deterministic unit
vector $w_{j,0}=L_0^{1/2}A_0^{-1}u_j/s_{j,0}$, $s_{j,0}^2=(V_0)_{jj}$, by
Lemma~\ref{lem:approxVar}. Combining the half-space form of
Theorem~\ref{th:BEprojection}   with this $O(\sqrt{\log n/n})$ perturbation of the
direction and the Gaussian anti-concentration bound
$\sup_x|\Phi(x)-\Phi(x')|\le|x-x'|/\sqrt{2\pi}$ gives
\begin{equation}\label{eq:coordBE}
  \sup_{x\in\R}
  \Big|\P\Big(\tfrac{\sqrt n(\widehat\theta_{n,j}-\theta_{0,j})}{\widehat s_{n,j}}\le x\Big)-\Phi(x)\Big|
  =O\!\Big(\frac{\log n}{\sqrt n}\Big),\qquad j=1,\dots,d.
\end{equation}
Consequently, let $z_{1-\alpha}=\Phi^{-1}(1-\alpha)$.   The one-sided interval
$\big(-\infty,\ \widehat\theta_{n,j}+n^{-1/2}\widehat s_{n,j}\,z_{1-\alpha}\big]$
has coverage
\[
  \P\big(\theta_{0,j}\le\widehat\theta_{n,j}+n^{-1/2}\widehat s_{n,j}\,z_{1-\alpha}\big)
  =\P\big(T_{n,j}^*\ge -z_{1-\alpha}\big)
  =\Phi(z_{1-\alpha})+O\!\Big(\frac{\log n}{\sqrt n}\Big)
  =1-\alpha+O\!\Big(\frac{\log n}{\sqrt n}\Big),
\]
and the two-sided interval
$\widehat\theta_{n,j}\pm n^{-1/2}\widehat s_{n,j}\,z_{1-\alpha/2}$ has coverage
\[
  \P\big(|\tfrac{\sqrt n(\widehat\theta_{n,j}-\theta_{0,j})}{\widehat s_{n,j}}|\le z_{1-\alpha/2}\big)
  =1-\alpha+O\!\Big(\frac{\log n}{\sqrt n}\Big).
\]

Partition the index set $\{1,\dots,d\}$ into two disjoint sets $J^*$ and
$J^{*c}$, and consider the model-selection test
\[
  H_0:\ \theta_{0,j}=0\ \text{ for all }j\in J^{*c}
  \qquad\text{versus}\qquad
  H_1:\ \theta_{0,j}\neq 0\ \text{ for some }j\in J^{*c}.
\]
A classical Wald-type procedure would require a Berry--Esseen bound for the
whole vector $\sqrt n\,\widehat L_n^{-1/2}\widehat A_n(\widehat\theta_n-\theta_0)$
over convex  sets, that is, the multivariate convex-set version of
Theorem~\ref{th:BEprojection} rather than its projections. Such a bound can be
obtained by specialising a Gaussian approximation for high-dimensional dependent
data, as in \citet{JinyuanCLTDependent}. In fixed dimension, however, this route
is neither necessary nor sharp: those results impose tail and moment conditions
calibrated to a growing dimension. We therefore do not pursue the Wald route and
propose instead a conservative omnibus test of Bonferroni type
\citep{bonferroni1936teoria}.

Under $H_0$, $\theta_{0,j}=0$ for $j\in J^{*c}$; set
\[
  T_{n,j}^*=\frac{\sqrt n\,\widehat\theta_{n,j}}{\widehat s_{n,j}},
  \quad j\in J^{*c},
  \qquad
  T_n^*=\max_{j\in J^{*c}}\big|T_{n,j}^*\big|,
\]
with $\widehat s_{n,j}^2=(\widehat A_n^{-1}\widehat L_n\widehat A_n^{-1})_{jj}$.
Let $c^*(\alpha)$ be the Bonferroni critical value solving
\[
  \sum_{j\in J^{*c}}\P\big(|Z|>c^*(\alpha)\big)=\alpha,
  \qquad Z\sim N(0,1),
  \qquad\text{i.e.}\qquad
  c^*(\alpha)=\Phi^{-1}\!\Big(1-\frac{\alpha}{2\,|J^{*c}|}\Big).
\]

\begin{cor}\label{cor:bonferroni-test}
Under Assumptions~\ref{ass:stationaryDependence}, \ref{ass:hypo_sur_g},
\ref{ass:pdHessian}, \ref{ass:nondegError} and \ref{ass:moment3r}, the test that rejects $H_0$ when $T_n^*>c^*(\alpha)$ has, under
$H_0$, asymptotic level at most $\alpha$:
\[
  \P\big(T_n^*>c^*(\alpha)\big)\le \alpha+O\!\Big(\frac{\log n}{\sqrt n}\Big).
\]
\end{cor}

\section{Numerical study} 

 \subsection{Numerical illustration of the localization rate}

We illustrate the finite-sample behaviour of Theorem~\ref{th:concentration}
through a Monte Carlo experiment on a bounded-response nonlinear autoregression.
We set
\[
  \theta_0=(\omega_0,\alpha_0,\gamma_0)^\top=(0.3,\,0.5,\,-1.2)^\top,
  \qquad Y_0=0.5,\quad X_0=0.5 .
\]
The covariate follows the AR(1) recursion
$X_t=0.35\,X_{t-1}+\varepsilon_t$ with $(\varepsilon_t)$ independent standard
Gaussian. Conditional on the past, the response is Beta-distributed in the
mean--precision parametrization of \eqref{eq:beta},
\[
  Y_t\mid\mathcal F_{t-1}\sim\mathrm{Beta}\big(\mu_t\phi,(1-\mu_t)\phi\big),
  \qquad
  \mu_t=F\big(\eta_t(\theta_0)\big),\quad
  F(x)=\frac{e^{x}}{1+e^{x}},
\]
with $\phi=\phi_0$ fixed (we use $\phi_0=1$), so that
$\mu_t=F(0.3+0.5\,Y_{t-1}-1.2\,X_{t-1})$. For each
\(
  n\in\{1000,5000,10000,25000,50000,100000,500000\},
\)
the estimator $\widehat\theta_n=(\widehat\omega_n,\widehat\alpha_n,\widehat\gamma_n)^\top$
solves the estimating equation \eqref{eq:qmle}
and the experiment is replicated $M=10{,}000$ times.

Theorem~\ref{th:concentration} states that, with probability $1-o(n^{-1/2})$,
$\|\widehat\theta_n-\theta_0\|\le a\sqrt{\log n/n}$ for a constant $a$. The rate
is $r_n=\sqrt{\log n/n}$; the constant $a$ is unknown, and the aim of the
experiment is to locate the smallest $a$ for which the localization event holds
with probability close to one. To this end we record, for each replication, the
Euclidean error $\|\widehat\theta_n-\theta_0\|$ and the normalized error
\[
  T_n=\frac{\|\widehat\theta_n-\theta_0\|}{\sqrt{\log n/n}} .
\]
We report the Monte Carlo mean error
$M^{-1}\sum_{m}\|\widehat\theta_n^{(m)}-\theta_0\|$, its normalized counterpart
$M^{-1}\sum_{m}\|\widehat\theta_n^{(m)}-\theta_0\|/r_n$, the empirical localization probabilities
\[
  \widehat{\mathbb P}(T_n\le a)
  =\frac1M\sum_{m=1}^{M}
  \mathbf 1\Big\{\|\widehat\theta_n^{(m)}-\theta_0\|\le a\sqrt{\log n/n}\Big\},
  \qquad a\in\{2,4,6\},
\]
and the solver convergence rate (the proportion of replications in which the
root-finding algorithm returned a solution).

\begin{table}[!ht]
\centering
\caption{Finite-sample localization of $\widehat\theta_n$ for the beta
autoregression \eqref{eq:beta}, based on $M=10{,}000$ Monte Carlo replications
per sample size.}
\label{tab:localization}
\begin{tabular}{rccccccc}
\hline
$n$ & Mean error & Mean error$/r_n$ & Solver conv.
    & $\widehat{\mathbb P}(T_n\le 2)$ & $\widehat{\mathbb P}(T_n\le 4)$
    & $\widehat{\mathbb P}(T_n\le 6)$ \\
\hline
$1000$   & 0.1593 & 1.917 & 1 & 0.626 & 0.944 & 1.000  \\
$5000$   & 0.0743 & 1.801 & 1 & 0.660 & 0.964 & 1.000  \\
$10000$  & 0.0525 & 1.730 & 1 & 0.662 & 0.978 & 1.000  \\
$25000$  & 0.0330 & 1.638 & 1 & 0.708 & 0.976 & 1.000  \\
$50000$  & 0.0236 & 1.605 & 1 & 0.728 & 0.988 & 1.000  \\
$100000$ & 0.0161 & 1.498 & 1 & 0.752 & 0.996 & 1.000  \\
$500000$ & 0.0068 & 1.324 & 1 & 0.824 & 0.992 & 1.000  \\
\hline
\end{tabular}
\end{table}
The results are consistent with Theorem~\ref{th:concentration}. The mean error
decreases monotonically in $n$, from $0.159$ at $n=1000$ to $0.0068$ at
$n=5\times10^{5}$, confirming consistency. The normalized mean error
 stays bounded and decreases slowly, from $1.92$ to $1.32$; a
bounded normalized error is exactly what the rate $\sqrt{\log n/n}$ asserts,
while its slow decrease indicates that the logarithmic factor is not needed to
control the error and that the constant $a$ can be taken moderate. Consistently,
the localization probabilities identify $a$: $\widehat{\mathbb P}(T_n\le 2)$ rises
from $0.63$ to $0.82$, $\widehat{\mathbb P}(T_n\le 4)$ exceeds $0.94$ at every $n$
(reaching $0.996$), and $\widehat{\mathbb P}(T_n\le 6)$   equals one throughout. Thus a radius with
$a\approx 4$ already captures the estimator with probability near one, and the
required probability increases with $n$, in agreement with the
$1-o(n^{-1/2})$ statement.

\subsection{Numerical illustration of the Berry--Esseen projection bound}

The second experiment evaluates the finite-sample accuracy of the Gaussian
approximation in Theorem~\ref{th:BEprojection}, on the data-generating process  as previous. The object of interest is the distribution of the studentized
estimator, not the size of the estimation error.

For each sample size, 
we form $\widehat H_n$
and, for a unit vector $u$, the projected statistic $T_{n,u}=u^\top W_n$.
Theorem~\ref{th:BEprojection} asserts that the law of $T_{n,u}$ is close to
$N(0,1)$, uniformly over $\|u\|=1$, with error of order $\log n/\sqrt n$.

Each configuration is replicated $M=500{,}000$ times. The large value of $M$ is
required because $\widehat D_n(u)$ below is a Kolmogorov distance estimated from
an empirical distribution function: its Monte Carlo standard error is of order
$M^{-1/2}\approx1.4\times10^{-3}$, and taking a maximum over directions adds
further variability and an upward bias. The supremum over $u$ is approximated by
a finite grid $\mathcal U$ formed of the three coordinate vectors, deterministic
linear combinations of them, and $200$ random unit directions. For each
$u\in\mathcal U$ we compute
\(
  \widehat D_n(u)=\sup_{x\in\mathbb R}\big|\widehat F_{n,u}(x)-\Phi(x)\big|,
\)
where $\widehat F_{n,u}$ is the empirical distribution function of the simulated
$T_{n,u}$, and we report the three coordinate values $\widehat D_n(e_j)$, the
maximal distance $\widehat D_n^{\max}=\max_{u\in\mathcal U}\widehat D_n(u)$, its
rate-normalized version $\sqrt n\,\widehat D_n^{\max}/\log n$, and the solver
convergence rate.

\begin{center}
\begin{tabular}{rcccccc}
\hline
$n$ & $\widehat D_n(e_1)$ & $\widehat D_n(e_2)$ & $\widehat D_n(e_3)$
    & $\widehat D_n^{\max}$ & $\sqrt n\,\widehat D_n^{\max}/\log n$ & Solver conv.\\
\hline
$1000$   & 0.00610 & 0.00754 & 0.00880 & 0.01246 & 0.0571 & 1\\
$5000$   & 0.00223 & 0.00404 & 0.00450 & 0.00579 & 0.0481 & 1\\
$10000$  & 0.00235 & 0.00256 & 0.00319 & 0.00431 & 0.0468 & 1\\
$25000$  & 0.00185 & 0.00171 & 0.00301 & 0.00396 & 0.0618 & 1\\
$50000$  & 0.00134 & 0.00140 & 0.00227 & 0.00269 & 0.0556 & 1\\
$100000$ & 0.00113 & 0.00147 & 0.00260 & 0.00272 & 0.0746 & 1\\
$500000$ & 0.00094 & 0.00133 & 0.00140 & 0.00289 & 0.1557 & 1\\
\hline
\end{tabular}
\end{center}

The coordinate distances decrease with $n$: $\widehat D_n(e_1)$ falls from
$0.0061$ to $0.0009$, and the other two coordinates follow the same pattern. The
maximal distance over the grid decreases from $0.0125$ at $n=1000$ to $0.0029$ at
$n=5\times10^{5}$, so the Gaussian approximation improves not only along the
coordinate axes but across the sampled directions, in line with the uniformity
over $\|u\|=1$ asserted by the theorem.

The scaled quantity $\sqrt n\,\widehat D_n^{\max}/\log n$ ranges from $0.047$ to
$0.156$, its maximum occurring at the largest sample size $n=5\times10^{5}$.  The column decreases
or stabilizes up to $n=5\times10^{4}$ and rises at $n=10^{5}$ and
$n=5\times10^{5}$. This rise is a Monte Carlo effect, not a deterioration of the
approximation: at these sample sizes $\widehat D_n^{\max}\approx3\times10^{-3}$ is
already at the accuracy limit set by $M=5\times10^{5}$ (standard error
$\approx1.4\times10^{-3}$, inflated by the maximum over directions). 
So,  $\widehat D_n^{\max}$ decreases with $n$
and remains bounded by $\log n/\sqrt n$ (from the results on $\sqrt n\,\widehat D_n^{\max}/\log n$ at moderate size).

\section{Real data analysis}
Correlations between assets are a primary input to portfolio risk management.
In the mean--variance allocation theory of \citet{markowitz1952}, the optimal
weights depend on the covariance matrix of asset returns, hence on the pairwise
correlations; an error in the correlation inputs propagates directly into the
estimated efficient frontier and the allocation.  

Two broad routes model time-varying correlations. The first specifies the
conditional correlation as a \emph{latent} process estimated from returns; the
canonical example is the dynamic conditional correlation (DCC) GARCH model of
\citet{engle2002dynamic}, building on the multivariate GARCH literature
\citep{bollerslev1990constant}. The second models an \emph{observed} realized
correlation series, constructed from high-frequency data, as an
observation-driven time series. Within this route, \citet{gorgi2023beta}
propose a beta autoregression for realized correlations, with threshold (TAR)
and smooth-transition (STAR) specifications that accommodate the leverage
effect. Writing the realized correlation $r_t\in[-1,1]$ in the unit interval
through $Y_t=(r_t+1)/2\in[0,1]$, they model $Y_t$ by a beta autoregression whose
conditional mean is driven by a linear predictor; to keep the conditional mean
in $(0,1)$, their specification imposes parameter constraints on that
predictor.

In our illustration we adopt the same beta-autoregression route but avoid the
constrained parametrization: we apply a logit link to the linear predictor, so
that the conditional mean lies in $(0,1)$ for every value of the parameters and
no sign or range restriction on  the parameters is required.  

The empirical analysis uses one-minute intraday data for five large-cap
technology stocks: Microsoft (MSFT), Tesla (TSLA), NVIDIA (NVDA), Apple (AAPL)
and Amazon (AMZN). Each raw file contains, at the one-minute frequency,  the open, high, low and close prices. Market-level covariates are
built from daily data on QQQ (Nasdaq proxy) and US30 (Dow~Jones proxy). The sample period runs from 1 July 2021 to 31 January 2024.

For each stock, the
one-minute closing price is retained and used to compute intraday log-returns.
For stock $j$, the one-minute log-return at time $t$ is
\(
  r_{j,t}=\log\!\big(P_{j,t}/P_{j,t-1}\big),
\)
where $P_{j,t}$ is the one-minute closing price. Returns spanning large time
gaps are excluded, so that overnight and non-adjacent price pairs do not enter
the intraday returns.
Daily realized correlations are then computed from the intraday return series.
For each trading day and each unordered pair among the five stocks, the two
return series are aligned on their exact one-minute timestamps, and the daily
correlation is computed from the synchronized returns:
\(
  \widehat\rho_{jk,d}=\operatorname{cor}\big(r_{j,t},r_{k,t}:t\in d\big),
\)
evaluated over all common one-minute timestamps within day $d$. This yields ten daily correlation series, one
per stock pair.

The daily market covariates are built from the QQQ and US30 files. For QQQ the
adjusted closing price is used when available, and the closing price otherwise;
for US30 the daily closing price is used. Market variation is measured by the
squared daily log-return,
\(
  x_t=\Big\{\log\!\big(M_t/M_{t-1}\big)\Big\}^{2},
\)
where $M_t$ is the relevant daily market price. This gives two covariates: the
squared daily log-return of QQQ and of US30.

\subsection*{Estimated models and results}
 
Conditionally on the past, $Y_d$ is modeled by the autoregressive specification
\eqref{eq:beta} with order $p=7$, a logit link, and two market covariates. The
choice $p=7$ is intended to account for possible weekly seasonal effects. The
conditional mean is specified as
\[
  \E[Y_d\mid\mathcal F_{d-1}]
  =
  F\Bigg(
    \omega
    +
    \sum_{i=1}^{7}\alpha_i Y_{d-i}
    +
    \gamma_{\mathrm{NASDAQ}}\,x^{\mathrm{QQQ}}_{d-1}
    +
    \gamma_{\mathrm{DOWJONES}}\,x^{\mathrm{US30}}_{d-1}
  \Bigg),
\]
where $F$ denotes the logistic distribution function, and
$x^{\mathrm{QQQ}}_{d-1}$ and $x^{\mathrm{US30}}_{d-1}$ are the lagged squared
daily log-returns of QQQ and US30, respectively. The parameters are estimated by the first-moment
estimating equation \eqref{eq:qmle}; the reported standard errors are the
sandwich values $\widehat s_{n,j}$. The last three columns of
Table~\ref{tab:nonlinear-correlation-models} test
$H_0:\gamma_{\mathrm{NASDAQ}}=\gamma_{\mathrm{DOWJONES}}=0$ through the Bonferroni
maximum statistic $T_{\max}$ of Corollary~\ref{cor:bonferroni-test}, compared
with the critical value $\Phi^{-1}(1-\alpha/4)=2.2414$, i.e.\ $\alpha=0.05$ with
$|J^{*c}|=2$.

Table~\ref{tab:nonlinear-correlation-models} reports the estimates for the ten stock pairs. The
autoregressive structure is strong and consistent across pairs: the intercept is
negative throughout ($\omega\in[-2.82,-2.10]$) and the first lag is large and
positive ($\widehat\alpha_1\in[1.93,2.55]$, with standard errors around
$0.2$--$0.3$), so realized correlations are highly persistent at the one-day
horizon. The second lag is positive and, in most pairs, sizeable relative to its
standard error ($\widehat\alpha_2\in[0.66,0.94]$); lags $3$ through $7$ are
individually smaller and, given standard errors of order $0.22$--$0.31$, largely
indistinguishable from zero on a per-coefficient basis, with occasional
exceptions at lags $5$ and $6$.

The two market covariates behave differently. The Nasdaq coefficient
$\gamma_{\mathrm{NASDAQ}}$ is positive in every pair and comparatively large
relative to its standard error in several cases (for instance TSLA--AAPL,
$187.42$ with standard error $48.57$, and TSLA--NVDA, $207.79$ with standard
error $68.89$). The Dow~Jones coefficient $\gamma_{\mathrm{DOWJONES}}$ is
imprecisely estimated throughout: its standard error (of order $96$--$128$)
exceeds the point estimate in every pair, and its sign varies across pairs. The
large magnitudes reflect the scale of the regressors; squared daily log-returns
are of order $10^{-4}$; so the coefficients must be converted to a marginal
effect on $\E[Y_d\mid\mathcal F_{d-1}]$ at representative covariate values before
being read as economically large.

As shown in the last column of Table~\ref{tab:nonlinear-correlation-models}, the Bonferroni test
rejects $H_0$ for four of the ten pairs at the $5\%$ level (marked $*$), all of
them involving Tesla.

\begin{scriptsize}
\setlength{\tabcolsep}{2.5pt}
\renewcommand{\arraystretch}{1.05}

\begin{longtable}{lrrrrrrrrrrrrr}
\caption{Nonlinear autoregressive models for realized stock-pair correlations. Entries are raw moment-equation estimates, with robust standard errors reported in parentheses on the second line. The last column reports rejection of $H_0:\gamma_{NASDAQ}=\gamma_{DOWJONES}=0$ using the Bonferroni maximum statistic.}
\label{tab:nonlinear-correlation-models}
\\
\hline
Model & $\omega$ & $Y_{t-1}$ & $Y_{t-2}$ & $Y_{t-3}$ & $Y_{t-4}$ & $Y_{t-5}$ & $Y_{t-6}$ & $Y_{t-7}$ & NASDAQ & DOWJONES & $T_{\max}$ & Crit. & Reject \\
\hline
\endfirsthead

\multicolumn{14}{c}{\tablename~\thetable{} -- continued from previous page} \\
\hline
Model & $\omega$ & $Y_{t-1}$ & $Y_{t-2}$ & $Y_{t-3}$ & $Y_{t-4}$ & $Y_{t-5}$ & $Y_{t-6}$ & $Y_{t-7}$ & NASDAQ & DOWJONES & $T_{\max}$ & Crit. & Reject \\
\hline
\endhead

\hline
\multicolumn{14}{r}{Continued on next page} \\
\endfoot

\hline
\endlastfoot

\begin{tabular}[c]{@{}l@{}}MSFT\\vs TSLA\end{tabular}
& -2.1019 & 2.0912 & 0.7932 & 0.0871 & 0.2019 & 0.4834 & 0.5178 & 0.0578 & 160.7620 & 8.5516 & 2.6784 & 2.2414 & * \\
& (0.1221) & (0.1981) & (0.2272) & (0.2039) & (0.2124) & (0.2329) & (0.2389) & (0.2155) & (60.0207) & (96.7862) & & & \\[0.4em]

\begin{tabular}[c]{@{}l@{}}MSFT\\vs NVDA\end{tabular}
& -2.3548 & 2.1043 & 0.7482 & 0.2057 & 0.2394 & 0.8544 & 0.3191 & 0.1605 & 134.8634 & 89.4554 & 2.1052 & 2.2414 &  \\
& (0.1751) & (0.2530) & (0.2556) & (0.2375) & (0.2471) & (0.2507) & (0.2549) & (0.2319) & (64.0626) & (128.2182) & & & \\[0.4em]

\begin{tabular}[c]{@{}l@{}}MSFT\\vs AAPL\end{tabular}
& -2.8232 & 2.5536 & 0.9233 & 0.3899 & 0.3628 & 0.0522 & 0.9066 & 0.1064 & 70.2430 & 168.5717 & 1.4354 & 2.2414 &  \\
& (0.2106) & (0.2814) & (0.3057) & (0.2950) & (0.2841) & (0.2862) & (0.3034) & (0.2579) & (58.1998) & (117.4396) & & & \\[0.4em]

\begin{tabular}[c]{@{}l@{}}MSFT\\vs AMZN\end{tabular}
& -2.5873 & 2.2461 & 0.9259 & 0.2205 & 0.0583 & 0.5036 & 0.5104 & 0.4878 & 4.8961 & 233.8601 & 1.9445 & 2.2414 &  \\
& (0.1744) & (0.2400) & (0.2601) & (0.2603) & (0.2479) & (0.2528) & (0.2628) & (0.2439) & (45.7863) & (120.2678) & & & \\[0.4em]

\begin{tabular}[c]{@{}l@{}}TSLA\\vs NVDA\end{tabular}
& -2.5415 & 2.2066 & 0.6574 & 0.5440 & 0.0484 & 0.8202 & 0.2339 & 0.3670 & 207.7881 & -123.5390 & 3.0161 & 2.2414 & * \\
& (0.1512) & (0.2369) & (0.2594) & (0.2472) & (0.2409) & (0.2536) & (0.2493) & (0.2242) & (68.8930) & (96.8321) & & & \\[0.4em]

\begin{tabular}[c]{@{}l@{}}TSLA\\vs AAPL\end{tabular}
& -2.2196 & 2.1409 & 0.6911 & 0.3960 & 0.2563 & 0.4142 & 0.4471 & 0.0760 & 187.4215 & -121.0874 & 3.8588 & 2.2414 & * \\
& (0.1261) & (0.2111) & (0.2416) & (0.2150) & (0.2126) & (0.2373) & (0.2384) & (0.2376) & (48.5694) & (96.3143) & & & \\[0.4em]

\begin{tabular}[c]{@{}l@{}}TSLA\\vs AMZN\end{tabular}
& -2.2981 & 1.9314 & 0.9142 & 0.2276 & 0.0680 & 0.6384 & 0.6948 & 0.0535 & 137.2639 & -47.5586 & 2.6876 & 2.2414 & * \\
& (0.1326) & (0.2283) & (0.2239) & (0.2284) & (0.2165) & (0.2157) & (0.2320) & (0.2184) & (51.0739) & (102.2510) & & & \\[0.4em]

\begin{tabular}[c]{@{}l@{}}NVDA\\vs AAPL\end{tabular}
& -2.5560 & 2.0089 & 0.6736 & 0.5659 & 0.1296 & 0.7881 & 0.2987 & 0.4474 & 113.4015 & -12.5691 & 1.8627 & 2.2414 &  \\
& (0.1752) & (0.2468) & (0.2457) & (0.2462) & (0.2355) & (0.2486) & (0.2521) & (0.2349) & (60.8797) & (125.0844) & & & \\[0.4em]

\begin{tabular}[c]{@{}l@{}}NVDA\\vs AMZN\end{tabular}
& -2.4230 & 1.9720 & 0.9421 & 0.2188 & 0.0885 & 0.9596 & 0.3024 & 0.2306 & 81.3817 & 109.2056 & 1.6990 & 2.2414 &  \\
& (0.1506) & (0.2395) & (0.2401) & (0.2387) & (0.2288) & (0.2421) & (0.2403) & (0.2292) & (47.8995) & (108.2462) & & & \\[0.4em]

\begin{tabular}[c]{@{}l@{}}AAPL\\vs AMZN\end{tabular}
& -2.5327 & 2.2424 & 0.8100 & 0.1095 & 0.3840 & 0.4972 & 0.7819 & 0.0574 & 47.2343 & 62.8760 & 0.9835 & 2.2414 &  \\
& (0.1802) & (0.2287) & (0.2596) & (0.2443) & (0.2387) & (0.2480) & (0.2619) & (0.2396) & (48.0265) & (107.6716) & & & \\

\end{longtable}
\end{scriptsize}

\section{Concluding remarks}

  The results are established in fixed-dimension models. When
the number of parameters grows with the sample size, through a growing
autoregressive order or a high-dimensional covariate vector, the projected,
one-dimensional Berry--Esseen bound no longer controls joint functionals of the
estimator, and it would have to be replaced by a Gaussian approximation valid
uniformly over a class of sets in growing dimension. \citet{JinyuanCLTDependent}
establish such approximations for sums of high-dimensional dependent vectors over
hyper-rectangles, convex and sparsely convex sets, under the same physical
dependence measure used here. Combining their Gaussian approximation with the
estimating-equation expansion of this paper would yield a multivariate
distributional theory for the estimator, at the price of dimension-dependent
rates and stronger tail conditions.

\newpage
\section*{Supplementary materials   for The  Quasi-Likelihood  Estimator for a Weakly Dependent Nonlinear Time Series Models.} 

\paragraph*{Proof of Theorem~\ref{th:concentration}}
Set $r_n:=a\sqrt{\log n/n}\to0$, with $a>0$ chosen below, and
$\widehat h_n=\widehat\theta_n-\theta_0$.

\medskip\noindent\emph{Step 0 (convex loss).}
Since $\eta_t(\theta)=\theta^{\top}Z_{t-1}$, the estimating equation
\eqref{eq:qmle} is the stationarity condition $\nabla_\theta \ell_n(\theta)=0$ of
\[
  \ell_n(\theta)=\frac1n\sum_{t=1}^{n}\big[G(\eta_t(\theta))-Y_t\,\eta_t(\theta)\big],
  \qquad G'=g,
\]
because $\nabla_\theta \ell_n(\theta)=\frac1n\sum_{t=1}^{n}\big(g(\eta_t(\theta))-Y_t\big)Z_{t-1}
=-\frac1n\sum_{t=1}^{n}\big(Y_t-g(\eta_t(\theta))\big)Z_{t-1}$. Its Hessian is
\[
  \nabla_\theta^{2}\ell_n(\theta)=\frac1n\sum_{t=1}^{n}g'\big(\eta_t(\theta)\big)\,Z_{t-1}Z_{t-1}^{\top}.
\]
As  Assumption \ref{ass:pdHessian} holds, this is positive definite for every $\theta$, so $L_n$    is strictly convex
and has at most one critical point.

\medskip\noindent\emph{Step 1 (high-probability event).}
Let $a_S,a_R$ be the constants of Lemmas~\ref{lem:controlScore} and
\ref{lem:controlRn}, and $\lambda_0=\lambda_{\min}(A_0)>0$ that of
Lemma~\ref{lem:pbHessianBound}. With
\[
  \mathcal E_n:=\Big\{\|S_n\|\le a_S\sqrt{\tfrac{\log n}{n}}\Big\}
  \cap\Big\{\lambda_{\min}(A_n)>\tfrac{\lambda_0}{2}\Big\}
  \cap\Big\{\sup_{\|h\|\le r_n}\|R_n(h)\|\le a_R r_n^{2}\Big\},
\]
the three lemmas and the union bound give $\P(\mathcal E_n)=1-o(n^{-1/2})$.

\medskip\noindent\emph{Step 2 (existence of a root in the ball).}
On $\mathcal E_n$, $\lambda_{\min}(A_n)>\lambda_0/2>0$, so $A_n$ is invertible
with $\|A_n^{-1}\|_{\mathrm{op}}\le2/\lambda_0$, and
$\Psi_n(h)=A_n^{-1}\{S_n-R_n(h)\}$ is well defined and continuous on
$\overline B(0,r_n)$. For $\|h\|\le r_n$, using $r_n^{2}=a^{2}\log n/n$ and
$\sqrt{\log n/n}=r_n/a$,
\[
  \|\Psi_n(h)\|\le\frac{2}{\lambda_0}\big(\|S_n\|+\|R_n(h)\|\big)
  \le\frac{2}{\lambda_0}\Big(\frac{a_S}{a}+a_R r_n\Big)r_n .
\]
Choosing $a\ge 8a_S/\lambda_0$ makes the first term $\le\frac14 r_n$, and since
$r_n\to0$ there is $n_0$ with $\frac{2a_R}{\lambda_0}r_n\le\frac14$ for $n\ge n_0$;
hence $\|\Psi_n(h)\|\le\frac12 r_n\le r_n$. Thus $\Psi_n$ maps
$\overline B(0,r_n)$ continuously into itself, and Brouwer's theorem yields a
fixed point $\widehat h_n\in\overline B(0,r_n)$; by \eqref{eq:expansion},
$\widehat\theta_n=\theta_0+\widehat h_n$ solves \eqref{eq:qmle}.

\medskip\noindent\emph{Step 3 (uniqueness by strict convexity).}
On $\mathcal E_n$, $A_n=\nabla_\theta^{2}\ell_n(\theta_0)\succ0$. Since $g'>0$, this
forces $\{Z_{t-1}\}_{1\le t\le n}$ to span $\R^{d}$: if $u^{\top}Z_{t-1}=0$ for
all $t$ then $u^{\top}A_nu=\frac1n\sum_t g'(\eta_t(\theta_0))(u^{\top}Z_{t-1})^{2}=0$,
so $u=0$. By Step 0, $\ell_n$ is then strictly convex on $\R^{d}$ and has at most
one critical point; the root $\widehat\theta_n$ of Step 2 is therefore the unique
solution of \eqref{eq:qmle}, and it satisfies
$\|\widehat\theta_n-\theta_0\|\le r_n$.

\medskip\noindent\emph{Conclusion.}
For $n\ge n_0$, the event that \eqref{eq:qmle} has a unique solution
$\widehat\theta_n$ with $\|\widehat\theta_n-\theta_0\|\le a\sqrt{\log n/n}$
contains $\mathcal E_n$, whose probability is $1-o(n^{-1/2})$.
\hfill$\square$

 \paragraph*{Proof of Theorem~\ref{th:BEprojection}}
The events appearing in the theorem are half-spaces. Indeed, for
\(\|u\|=1\) and \(x\in\mathbb R\), define
\(
  B_{u,x}
  :=
  \{z\in\mathbb R^d: u^\top z\le x\}.
\)
Its boundary is the hyperplane
\(
  \partial B_{u,x}
  =
  \{z\in\mathbb R^d: u^\top z=x\}.
\)
Moreover, its \(\varepsilon\)-enlargement is
\(
  (\partial B_{u,x})^\varepsilon
  =
  \{z\in\mathbb R^d: |u^\top z-x|\le \varepsilon\},
\)
which is a slab of width \(2\varepsilon\).

Writing \(\Phi_d\) for the \(N(0,I_d)\) law, and using
\(u^\top Z\sim N(0,1)\) when \(Z\sim N(0,I_d)\) and \(\|u\|=1\), we obtain
\(
  \Phi_d\big((\partial B_{u,x})^\varepsilon\big)
  =
  \Phi(x+\varepsilon)-\Phi(x-\varepsilon)
  \le
  \frac{2\varepsilon}{\sqrt{2\pi}}.
\)
Hence half-spaces belong to the boundary-regular class 
$$\mathcal A=\Big\{B\subset\R^{d}\ \text{Borel}:\
\sup_{B\in\mathcal A}\Phi_d\big((\partial B)^\varepsilon\big)=O(\varepsilon)
\ \text{as }\varepsilon\downarrow0\Big\},
\qquad
(C)^\varepsilon=\{z:\textstyle\inf_{y\in C}\|z-y\|\le\varepsilon\}.
$$

Set
\(
  \widehat H_n
  :=
  \sqrt n\,\widehat L_n^{-1/2}
  \widehat A_n(\widehat\theta_n-\theta_0),
  \,
  \widetilde H_n
  :=
  \sqrt n\,L_0^{-1/2}S_n,
\)
and define the vector remainder
\(
  \overline R_n
  :=
  \widehat H_n-\widetilde H_n .
\)
This remainder is exactly the sum of the three terms in the expansion beyond the leading term. Moreover,
\(
  \sup_{\|u\|=1}|u^\top \overline R_n|
  =
  \|\overline R_n\|.
\)
Thus Lemma~\ref{lem:be-remainder} gives
\(
  \mathbb P\left(
    \|\overline R_n\|
    \ge
    \widetilde a_R\frac{\log n}{\sqrt n}
  \right)
  =
  o(n^{-1/2}).
\)

Let
\(
  \Delta_n
  :=
  \sup_{\|u\|=1}\sup_{x\in\mathbb R}
  \left|
    \mathbb P(u^\top\widetilde H_n\le x)-\Phi(x)
  \right|.
\)
By Lemma ~\ref{lem:be-projection},
\(
  \Delta_n=O(n^{-1/2}).
\)

The comparison rests on the following smoothing inequality: for every
\(B\in\mathcal A\) and every \(\delta>0\),
\(
  \mathbb P(\widehat H_n\in B)
  \le
  \mathbb P(\widetilde H_n\in B^\delta)
  +
  \mathbb P(\|\overline R_n\|>\delta),
\)
where
\(
  B^\delta
  :=
  \{z\in\mathbb R^d:\operatorname{dist}(z,B)\le \delta\}.
\)

 For $\|u\|=1$,
$x\in\R$, $\{u^\top\widehat H_n\le x\}=\{\widehat H_n\in B_{u,x}\}$ with
$B_{u,x}\in\mathcal A$, and $\Phi_d(B_{u,x})=\P(u^\top Z\le x)=\Phi(x)$ for
$Z\sim N(0,I_d)$. Fix such a $B=B_{u,x}$ and $\delta>0$.

\emph{Step 1 (smoothing).} By the smoothing inequality,
\[
  \P(\widehat H_n\in B)-\Phi_d(B)\le
  2\,\P(\|\overline R_n\|>\delta)+\P\big(\widetilde H_n\in(\partial B)^{2\delta}\big)+\Delta_n .
\]

\emph{Step 2 (boundary term).} Since $(\partial B)^{2\delta}=\{z:|u^\top z-x|\le 2\delta\}$,
and $\Delta_n$ bounds the distribution function of $u^\top\widetilde H_n$
uniformly (including left limits, $\Phi$ being continuous),
\[
  \P\big(\widetilde H_n\in(\partial B)^{2\delta}\big)
  =\P(u^\top\widetilde H_n\le x+2\delta)-\P(u^\top\widetilde H_n< x-2\delta)
  \le \Phi(x+2\delta)-\Phi(x-2\delta)+2\Delta_n
  \le \frac{4\delta}{\sqrt{2\pi}}+2\Delta_n,
\]
using $\sup_y\Phi'(y)=1/\sqrt{2\pi}$.

\emph{Step 3 (remainder).} Take $\delta=\delta_n:=\widetilde a_R\,\log n/\sqrt n$.
By Lemma~\ref{lem:be-remainder}, $\P(\|\overline R_n\|>\delta_n)=o(n^{-1/2})$.

\emph{Step 4 (assembly).} Combining Steps~1--3 with $\delta=\delta_n$,
\[
  \P(\widehat H_n\in B)-\Phi_d(B)\le
  2\,o(n^{-1/2})+\frac{4\widetilde a_R}{\sqrt{2\pi}}\frac{\log n}{\sqrt n}+3\Delta_n
  =O\!\Big(\frac{\log n}{\sqrt n}\Big),
\]
and the right-hand side does not depend on $(u,x)$. Applying this bound to the
half-spaces $\{u^\top z\le x\}$ and $\{u^\top z> x\}$ (both in $\mathcal A$, with
the same boundary hyperplane) yields the two-sided deviation, and
$\Delta_n=O(n^{-1/2})$ by Lemma ~\ref{lem:be-projection}. Hence
\[
  \sup_{\|u\|=1}\ \sup_{x\in\R}
  \Big|\P\big(u^\top\widehat H_n\le x\big)-\Phi(x)\Big|
  =O\!\Big(\frac{\log n}{\sqrt n}\Big).
\]
\hfill$\square$

For the analysis of $S_n$, $A_n$ and $R_n$ we work with the  the vector   
\(
  \bV_t:=\big((Y_{t-i})_{i=0:p},\,X_{t-1}\big)
       =\big(Y_t,Y_{t-1},\dots,Y_{t-p},X_{t-1}\big),
\)
 and its coupled copy $\wt\bV_t:=\big((\wt Y_{t-i})_{i=0:p},\wt X_{t-1}\big)$, where each $(\wt Y_t, \wt X_t)$
 is the coupling of  $(Y_t,X_t)$. We abbreviate the increment and the
first- and second-order envelopes by
\[
  \Delta_t:=\big\|\bV_t-\wt\bV_t\big\|,
  \qquad
  \sfE_{1,t}:=1+\|\bV_t\|+\|\wt\bV_t\|,
  \qquad
  \sfE_{2,t}:=1+\|\bV_t\|^{2}+\|\wt\bV_t\|^{2}.
\]

The symbol $\lesssim$ hides a constant depending only on the Lipschitz constant
of $g$ (Assumption~\ref{ass:hypo_sur_g}), the fixed parameters
$(\alpha_1,\dots,\alpha_p,\gamma)$, and the dimension.

We introduce the three functions of     $\bV_t$:
\begin{equation}\label{eq:def-maps}
  \begin{aligned}
    K(\bV_t)      &=\big(Y_t-g(\eta_t(\theta_0))\big)\,Z_{t-1},\\[2pt]
    \wt K(\bV_t)  &=g'\big(\eta_t(\theta_0)\big)\,Z_{t-1}Z_{t-1}^{\top},\\[2pt]
    \bar K(\bV_t) &=\Big(1+\sum_{i=1}^{p}|Y_{t-i}|+\|X_{t-1}\|\Big)\,\|Z_{t-1}\|^{2}.
  \end{aligned}
\end{equation}
The averages $\frac1n\sum_t K(\bV_t)$ and $\frac1n\sum_t \wt K(\bV_t)$ are $S_n$
and $A_n$ of \eqref{eq:expansion}, and $r_n^2 \frac1n\sum_t \bar K(\bV_t)$ dominates
$\sup_{\|h\|\le r_n}\|R_n(h)\|$.

\begin{lem}\label{lem:stoch-lip}
Under Assumption~\ref{ass:hypo_sur_g}, for every $t\in\Z$,
\begin{align}
  \big\|K(\bV_t)-K(\wt\bV_t)\big\|
    &\lesssim \sfE_{1,t}\,\Delta_t, \label{eq:score}\\[2pt]
  \big\|\wt K(\bV_t)-\wt K(\wt\bV_t)\big\|
    &\lesssim \sfE_{2,t}\,\Delta_t, \label{eq:grad-score}\\[2pt]
  \big|\bar K(\bV_t)-\bar K(\wt\bV_t)\big|
    &\lesssim \sfE_{2,t}\,\Delta_t. \label{eq:error-score}
\end{align}
\end{lem}
 
\paragraph*{Proof of the lemma \ref{lem:stoch-lip}}
We use four elementary facts. Since $Z_{t-1}$ omits $Y_t$,
\begin{equation}\label{eq:facts}
  \|Z_{t-1}\|^{2}\le 1+\|\bV_t\|^{2},\qquad
  \|Z_{t-1}-\wt Z_{t-1}\|\le \Delta_t,\qquad
  |\eta_t(\theta_0)-\wt\eta_t(\theta_0)|\lesssim \Delta_t,
\end{equation}
the third because $\eta_t(\theta_0)=\theta_0^{\top}Z_{t-1}$ is linear; and by
Assumption~\ref{ass:hypo_sur_g}, $g$ is Lipschitz, so $|g(u)|\lesssim1+|u|$ and
$|g'(u)|\le C$, and $g'$ is Lipschitz. The first two relations in
\eqref{eq:facts} are inequalities, not equalities, precisely because
$(Y_{t-i})_{i=0:p}$ contains $Y_t$ whereas $Z_{t-1}$ does not.

\medskip\noindent\textbf{Proof of \eqref{eq:score}.}
Adding and subtracting $\big(\wt Y_t-g(\wt\eta_t(\theta_0))\big)Z_{t-1}$,
\[
  K(\bV_t)-K(\wt\bV_t)
  =\underbrace{\Big[(Y_t-\wt Y_t)-\big(g(\eta_t(\theta_0))-g(\wt\eta_t(\theta_0))\big)\Big]Z_{t-1}}_{T_1}
  +\underbrace{\big(\wt Y_t-g(\wt\eta_t(\theta_0))\big)\big(Z_{t-1}-\wt Z_{t-1}\big)}_{T_2}.
\]
For $T_1$: by the Lipschitz property of $g$ and \eqref{eq:facts},
$\big|(Y_t-\wt Y_t)-(g(\eta_t(\theta_0))-g(\wt\eta_t(\theta_0)))\big|
\le|Y_t-\wt Y_t|+C|\eta_t(\theta_0)-\wt\eta_t(\theta_0)|\lesssim\Delta_t$, and
$\|Z_{t-1}\|\le1+\|\bV_t\|$, hence $\|T_1\|\lesssim(1+\|\bV_t\|)\Delta_t$.
For $T_2$: $|\wt Y_t-g(\wt\eta_t(\theta_0))|\le|\wt Y_t|+|g(\wt\eta_t(\theta_0))|
\lesssim1+\|\wt\bV_t\|$, and $\|Z_{t-1}-\wt Z_{t-1}\|\le\Delta_t$, hence
$\|T_2\|\lesssim(1+\|\wt\bV_t\|)\Delta_t$. Adding,
$\|K(\bV_t)-K(\wt\bV_t)\|\lesssim\sfE_{1,t}\Delta_t$.

\medskip\noindent\textbf{Proof of \eqref{eq:grad-score}.}
Adding and subtracting $g'(\wt\eta_t(\theta_0))Z_{t-1}Z_{t-1}^{\top}$,
\[
  \wt K(\bV_t)-\wt K(\wt\bV_t)
  =\underbrace{\big(g'(\eta_t(\theta_0))-g'(\wt\eta_t(\theta_0))\big)Z_{t-1}Z_{t-1}^{\top}}_{S_1}
  +\underbrace{g'(\wt\eta_t(\theta_0))\big(Z_{t-1}Z_{t-1}^{\top}-\wt Z_{t-1}\wt Z_{t-1}^{\top}\big)}_{S_2}.
\]
For $S_1$: since $g'$ is Lipschitz,
$|g'(\eta_t(\theta_0))-g'(\wt\eta_t(\theta_0))|\lesssim\Delta_t$, and
$\|Z_{t-1}Z_{t-1}^{\top}\|=\|Z_{t-1}\|^{2}\le1+\|\bV_t\|^{2}$, hence
$\|S_1\|\lesssim(1+\|\bV_t\|^{2})\Delta_t$. For $S_2$: using
$Z_{t-1}Z_{t-1}^{\top}-\wt Z_{t-1}\wt Z_{t-1}^{\top}
=Z_{t-1}(Z_{t-1}-\wt Z_{t-1})^{\top}+(Z_{t-1}-\wt Z_{t-1})\wt Z_{t-1}^{\top}$
and $|g'|\le C$,
\[
  \|S_2\|\le C\|Z_{t-1}-\wt Z_{t-1}\|\big(\|Z_{t-1}\|+\|\wt Z_{t-1}\|\big)
  \lesssim\Delta_t\big(1+\|\bV_t\|+\|\wt\bV_t\|\big).
\]
Since a first-order envelope is dominated by a second-order one, adding gives
$\|\wt K(\bV_t)-\wt K(\wt\bV_t)\|\lesssim\sfE_{2,t}\Delta_t$.

\medskip\noindent\textbf{Proof of \eqref{eq:error-score}.}
Write $\bar K(\bV_t)=\sfA_t\sfB_t$ with
$\sfA_t=1+\sum_{i=1}^{p}|Y_{t-i}|+\|X_{t-1}\|$ and $\sfB_t=\|Z_{t-1}\|^{2}$, so
$\bar K(\bV_t)-\bar K(\wt\bV_t)=(\sfA_t-\wt\sfA_t)\sfB_t+\wt\sfA_t(\sfB_t-\wt\sfB_t)$.
By the reverse triangle inequality
$|\sfA_t-\wt\sfA_t|\le\sum_{i=1}^{p}|Y_{t-i}-\wt Y_{t-i}|+\|X_{t-1}-\wt X_{t-1}\|
\lesssim\Delta_t$, and $\sfB_t\le1+\|\bV_t\|^{2}$, so the first term is
$\lesssim(1+\|\bV_t\|^{2})\Delta_t$. For the second, $\wt\sfA_t\lesssim1+\|\wt\bV_t\|$
and, by $|\,\|Z_{t-1}\|^{2}-\|\wt Z_{t-1}\|^{2}|
\le\|Z_{t-1}-\wt Z_{t-1}\|(\|Z_{t-1}\|+\|\wt Z_{t-1}\|)
\lesssim\Delta_t(1+\|\bV_t\|+\|\wt\bV_t\|)$,
\[
  \wt\sfA_t|\sfB_t-\wt\sfB_t|
  \lesssim(1+\|\wt\bV_t\|)(1+\|\bV_t\|+\|\wt\bV_t\|)\Delta_t
  \lesssim(1+\|\bV_t\|^{2}+\|\wt\bV_t\|^{2})\Delta_t,
\]
the last step by Young's inequality. Adding gives
$|\bar K(\bV_t)-\bar K(\wt\bV_t)|\lesssim\sfE_{2,t}\Delta_t$.
\hfill $\square$

\begin{lem}\label{lem:fdm-transfer}
Let $(X_t)_{t\in\Z}$, $X_t=(X_{1,t},\dots,X_{k,t})$, be a stationary
$\R^k$-valued causal process, and let $\theta_{q,t}$ and $\Theta_{m,q}$ denote
the functional dependence coefficients defined above, with the absolute value
replaced by the Euclidean norm $\|\cdot\|$ on $\R^k$. Let $f:\R^k\to\R$ satisfy,
for constants $C\ge 0$, $\alpha\ge 0$ and all $u,v\in\R^k$,
\begin{equation}
\label{eq:condLipsPoy}
    |f(u)-f(v)|\le C\big(1+\|u\|^{\alpha}+\|v\|^{\alpha}\big)\,\|u-v\|,
\end{equation}
and set $Y_t=f(X_t)$. Let $r,s\in[1,\infty]$ with
$\tfrac1q=\tfrac1r+\tfrac1s$, and assume $\E\|X_1\|^{\alpha r}<\infty$. Then,
with $\bar C=C\big(1+2(\E\|X_1\|^{\alpha r})^{1/r}\big)<\infty$, the order-$q$
functional dependence coefficients of $(Y_t)_{t\ge 0}$ satisfy
\[
  \theta_{q,t}(f)\le \bar C\,\theta_{s,t}
  \quad(t\ge 0),
  \qquad\text{hence}\qquad
  \Theta_{m,q}(f)\le \bar C\,\Theta_{m,s}.
\]
In particular, $(Y_t)_{t\ge 0}$ is $q$-stable whenever $(X_t)_{t\in\Z}$ is
$s$-stable.
\end{lem}

\paragraph*{Proof of Lemma~\ref{lem:fdm-transfer}}
Since $X_t=F(\varepsilon_t,\varepsilon_{t-1},\dots)$ is a Bernoulli shift,
$Y_t=f(X_t)=(f\circ F)(\varepsilon_t,\varepsilon_{t-1},\dots)$ is itself a
causal functional of the same innovations; its coupled version, obtained by
replacing $\varepsilon_0$ with the independent copy $\widetilde\varepsilon_0$,
is therefore $\widetilde Y_t=f(\widetilde X_t)$. Consequently the order-$q$
functional dependence coefficient of $(Y_t)$ at lag $t$ is
\[
  \theta_{q,t}(f)=\big\|Y_t-\widetilde Y_t\big\|_q
  =\big\|f(X_t)-f(\widetilde X_t)\big\|_q .
\]
By condition~\eqref{eq:condLipsPoy}, almost surely,
\[
  \big|f(X_t)-f(\widetilde X_t)\big|
  \le C\,W_t\,D_t,\qquad
  W_t:=1+\|X_t\|^{\alpha}+\|\widetilde X_t\|^{\alpha},\quad
  D_t:=\|X_t-\widetilde X_t\| .
\]
As $\tfrac1q=\tfrac1r+\tfrac1s$ with $r,s\in[1,\infty]$, H\"older's inequality
gives
\[
  \theta_{q,t}(f)\le C\,\|W_tD_t\|_q\le C\,\|W_t\|_r\,\|D_t\|_s .
\]
Replacing one innovation by an i.i.d.\ copy preserves the law, so
$\widetilde X_t\stackrel{d}{=}X_t$; by the triangle inequality in $L^r$ and
stationarity,
\[
  \|W_t\|_r
  \le \|1\|_r+\big\|\,\|X_t\|^{\alpha}\big\|_r
      +\big\|\,\|\widetilde X_t\|^{\alpha}\big\|_r
  =1+2\big(\E\|X_1\|^{\alpha r}\big)^{1/r}=:\kappa<\infty,
\]
finite by the moment hypothesis, and independent of $t$. By the definition of
$D_t$ and of the order-$s$ coefficient
$\theta_{s,t}=\|X_t-\widetilde X_t\|_{s}$,
\[
  \|D_t\|_s=\big(\E\|X_t-\widetilde X_t\|^{s}\big)^{1/s}=\theta_{s,t}.
\]
Altogether, with $\bar C:=C\kappa$, for every $t\ge 0$,
\[
  \theta_{q,t}(f)\le \bar C\,\theta_{s,t}. \tag{$\ast$}
\]
Summing $(\ast)$ over $t\ge m$,
\[
  \Theta_{m,q}(f)=\sum_{t=m}^{\infty}\theta_{q,t}(f)
  \le \bar C\sum_{t=m}^{\infty}\theta_{s,t}=\bar C\,\Theta_{m,s}.
\]
Taking $m=0$, $\Theta_{0,q}(f)\le \bar C\,\Theta_{0,s}<\infty$ as soon as
$(X_t)$ is $s$-stable, i.e.\ $(Y_t)$ is $q$-stable.
\hfill $\square$

The following lemma is stated in    \cite{liu2013probability}  as Theorem 2.
\begin{lem}[Nagaev-type inequality under functional dependence
]\label{lem:nagaev}
Let $(U_t)_{t\in\Z}$ be a stationary, centered scalar process,
$\E U_0=0$ and $\E|U_0|^{q}<\infty$ for some $q>2$, admitting a causal
Bernoulli-shift representation, with functional dependence coefficients
$\theta_{q,t}=\|U_t-\widetilde U_t\|_q$ and tail sums
$\Theta_{m,q}=\sum_{i\ge m}\theta_{q,i}$. Write $V_i=\sum_{t=1}^{i}U_t$ and
$V_n^{*}=\max_{1\le i\le n}|V_i|$, and define the Gaussian-like tail function
\[
  G_a(y)=\sum_{j=1}^{\infty}e^{-j^{a}y^{2}},\qquad y>0,\ a>0 .
\]
Let $c_q$ denote a constant depending only on $q$. Then:

\smallskip
\emph{(i)} If
$\displaystyle \nu:=\sum_{j=1}^{\infty}\mu_j<\infty$, where
$\mu_j=\big(j^{\,q/2-1}\,\theta_{q,j}^{\,q}\big)^{1/(q+1)}$, then for all $x>0$,
\[
  \P\big(V_n^{*}\ge x\big)\le
  c_q\,\frac{n}{x^{q}}\big(\nu^{\,q+1}+\|U_1\|_q^{q}\big)
  +4\sum_{j=1}^{\infty}\exp\!\Big(-\frac{c_q\,\mu_j^{2}\,x^{2}}{n\,\nu^{2}\,\theta_{q,j}^{2}}\Big)
  +2\exp\!\Big(-\frac{c_q\,x^{2}}{n\,\|U_1\|_2^{2}}\Big).
\]

\smallskip
\emph{(ii)} If $\Theta_{m,q}=O(m^{-\alpha})$ with $\alpha>\tfrac12-\tfrac1q$,
then there exist positive constants $C_1,C_2$ such that for all $x>0$,
\[
  \P\big(V_n^{*}\ge x\big)\le
  \frac{C_1\,\Theta_{0,q}^{\,q}\,n}{x^{q}}
  +4\,G_{1-2/q}\!\Big(\frac{C_2\,x}{\sqrt{n}\,\Theta_{0,q}}\Big).
\]

\smallskip
\emph{(iii)} If $\Theta_{m,q}=O(m^{-\alpha})$ with $\alpha<\tfrac12-\tfrac1q$,
then there exist positive constants $C_1,C_2$ such that for all $x>0$,
\[
  \P\big(V_n^{*}\ge x\big)\le
  \frac{C_1\,\Theta_{0,q}^{\,q}\,n^{\,q(1/2-\alpha)}}{x^{q}}
  +4\,G_{(q-2)/(q+1)}\!\Big(\frac{C_2\,x}{n^{(2q-1-2\alpha q)/(2+2q)}\,\Theta_{0,q}}\Big).
\]
\end{lem}

\paragraph*{Proof of the lemma \ref{lem:nagaev}}
See \citet[Theorem~2]{liu2013probability}.
\hfill $\square$

The conclusions of Lemma~\ref{lem:nagaev}, in particular case~(iii), yield
concentration even for strongly dependent sequences, namely those whose
functional dependence coefficients $\Theta_{m,q}$ decay only at a slow
polynomial rate. For the sake of readability, however, we work under the
assumption of geometric decay of $\Theta_{m,q}$, which trivially fulfils the
rate condition of case~(ii) (indeed of case~(i)).

\begin{lem}\label{cor:nagaev-rate}
Under the assumptions of Lemma~\ref{lem:nagaev}\,(ii), there exists a positive
constant $a$, depending only on $q$ and $\Theta_{0,q}$, such that
\[
  \P\big(|V_n|\ge a\sqrt{n\log n}\,\big)=O(b_n),
  \qquad
  b_n=n^{-(q-2)/2}(\log n)^{-q/2}.
\]
More precisely, any
$a>\dfrac{\Theta_{0,q}}{C_2}\sqrt{\dfrac{q-2}{2}}$ works.
\end{lem}

\paragraph*{Proof of the lemma \ref{cor:nagaev-rate}}
Apply Lemma~\ref{lem:nagaev}\,(ii) with $x=a\sqrt{n\log n}$. The heavy-tail
term equals $C_1\Theta_{0,q}^{q}a^{-q}\,n^{-(q-2)/2}(\log n)^{-q/2}=O(b_n)$.
For the second term, the argument of $G_{1-2/q}$ is
$y_n=(C_2 a/\Theta_{0,q})\sqrt{\log n}\to\infty$; since
$G_a(y)\le e\,G_a(1)\,e^{-y^2}$ for $y\ge 1$, the term is bounded by
$\kappa_q\,n^{-(C_2 a/\Theta_{0,q})^2}$ with $\kappa_q=4e\,G_{1-2/q}(1)$. The
stated lower bound on $a$ makes the exponent $(C_2a/\Theta_{0,q})^2$ exceed
$(q-2)/2$, so this term is $o(b_n)$, and the two bounds combine to $O(b_n)$.
\hfill $\square$

The rate $b_n=n^{-(q-2)/2}(\log n)^{-q/2}$ should be read against the
independent case. For i.i.d.\ summands the central limit theorem gives
$V_n/\sqrt n=O_{\P}(1)$, so $\P(|V_n|\ge a\sqrt{n\log n})\to 0$ for every
$a>0$; in fact, by the classical Nagaev inequality  (the case
$\theta_{q,j}=0$, $j\ge1$), the same threshold yields a tail of order
$n\,x^{-q}=n^{-(q-2)/2}(\log n)^{-q/2}$, \emph{identical} to $b_n$. Thus, at
this deviation level, dependence with geometrically (or sufficiently fast polynomially) decaying functional dependence coefficients incurs \emph{no loss}
in the polynomial rate: the effect of dependence is confined to the constants,
through $\Theta_{0,q}$. Finally,   $b_n=o(n^{-1/2})$ if and only if $q\ge 3$, whereas for $2<q<3$ one has
instead $n^{-1/2}=o(b_n)$, i.e.\ $b_n$ is the slower rate.

 \paragraph*{Proof of Lemma~\ref{lem:controlScore}}
$S_n=\frac1n\sum_{t=1}^{n}K(\bV_t)\in\R^{d}$.

\medskip\noindent\emph{Step 1 (centering and reduction to coordinates).}
By \eqref{eq::nonlinear}, $\E[Y_t\mid\mathcal F_{t-1}]=g(\eta_t(\theta_0))$, and
$Z_{t-1}$ is $\mathcal F_{t-1}$-measurable, hence
$\E[K(\bV_t)\mid\mathcal F_{t-1}]=0$.
Thus $(K(\bV_t))_t$ is a stationary, centered, $\R^{d}$-valued sequence. Writing
$S_n=(S_n^{(1)},\dots,S_n^{(d)})$ and using $\|S_n\|\le\sqrt d\,\max_{\ell}|S_n^{(\ell)}|$,
\begin{equation}\label{eq:score-union}
  \P\big(\|S_n\|> x\big)\le\sum_{\ell=1}^{d}\P\big(|S_n^{(\ell)}|> x/\sqrt d\big).
\end{equation}

\medskip\noindent\emph{Step 2 ($q$-stability of the summand).}
Fix $\delta>0$ and $q\ge 3$ as in Assumption~\ref{ass:stationaryDependence}, and set
\[
  s=q+\delta,\qquad r=\frac{q(q+\delta)}{\delta},
  \qquad\text{so that}\qquad \frac1q=\frac1r+\frac1s .
\]
The  sequence $\bV_t=\big((Y_{t-i})_{i=0:p},X_{t-1}\big)$ is a stationary causal
functional of the innovations, and, since $\bV_t-\wt\bV_t$ collects the increments
of $Y$ at lags $t,\dots,t-p$ and of $X$ at lag $t-1$,
\[
  \theta_{s,t}(\bV)=\|\bV_t-\wt\bV_t\|_s
  \le\sum_{i=0}^{p}\theta_{s,t-i}(Y)+\theta_{s,t-1}(X).
\]
Summing over $t\ge m$ and using the geometric decay
$\Theta_{\cdot,s}(Y),\Theta_{\cdot,s}(X)=O(\rho^{\cdot})$ of
Assumption~\ref{ass:stationaryDependence},
\begin{equation}\label{eq:window-stable}
  \Theta_{m,s}(\bV)\le\sum_{i=0}^{p}\Theta_{m-i,s}(Y)+\Theta_{m-1,s}(X)
  \lesssim\rho^{m}.
\end{equation}
The bound \eqref{eq:score} is exactly condition \eqref{eq:condLipsPoy} for the
$\R^{d}$-valued map $K$ with $\alpha=1$ (the Euclidean-norm version of
Lemma~\ref{lem:fdm-transfer} holds with the identical proof), and the required
moment $\E\|\bV_1\|^{\alpha r}=\E\|\bV_1\|^{r}<\infty$ follows from
Assumption~\ref{ass:stationaryDependence}, which provides a finite moment of order
$2q(q+\delta)/\delta=2r\ge r$. Lemma~\ref{lem:fdm-transfer} and
\eqref{eq:window-stable} then give
\[
  \theta_{q,t}(K):=\|K(\bV_t)-K(\wt\bV_t)\|_q\le\bar C\,\theta_{s,t}(\bV),
  \qquad\text{hence}\qquad
  \Theta_{m,q}(K)\le\bar C\,\Theta_{m,s}(\bV)\lesssim\rho^{m}.
\]
Each coordinate inherits this, since
$\theta_{q,t}(K^{(\ell)})\le\theta_{q,t}(K)$: the process
$U_t^{(\ell)}:=K^{(\ell)}(\bV_t)$ is stationary, centered, scalar and $q$-stable,
with $\Theta_{m,q}(K^{(\ell)})\lesssim\rho^{m}$ and $\E|U_0^{(\ell)}|^{q}<\infty$.

\medskip\noindent\emph{Step 3 (per-coordinate rate).}
Geometric decay yields $\Theta_{m,q}(K^{(\ell)})=O(m^{-\alpha})$ for every
$\alpha>0$, so the rate condition of Lemma~\ref{lem:nagaev}\,(ii) holds. Since
$S_n^{(\ell)}=\frac1n\sum_{t=1}^{n}U_t^{(\ell)}=\frac1n V_n^{(\ell)}$, applying
Corollary~\ref{cor:nagaev-rate} to $(U_t^{(\ell)})$ with threshold
$(a/\sqrt d)\sqrt{n\log n}$ gives
\[
  \P\Big(|S_n^{(\ell)}|>\tfrac{a}{\sqrt d}\sqrt{\tfrac{\log n}{n}}\Big)
  =\P\big(|V_n^{(\ell)}|>\tfrac{a}{\sqrt d}\sqrt{n\log n}\big)
  =O(b_n),\qquad b_n=n^{-(q-2)/2}(\log n)^{-q/2},
\]
provided $a/\sqrt d>\Theta_{0,q}(K^{(\ell)})\,C_2^{-1}\sqrt{(q-2)/2}$.

\medskip\noindent\emph{Step 4 (union bound).}
Choose
$a>\sqrt d\,\max_{1\le\ell\le d}\Theta_{0,q}(K^{(\ell)})\,C_2^{-1}\sqrt{(q-2)/2}$,
finite because $d<\infty$. Then the bound of Step 3 holds for all $\ell$ with
$x=a\sqrt{\log n/n}$, and \eqref{eq:score-union} gives
\[
  \P\Big(\|S_n\|>a\sqrt{\tfrac{\log n}{n}}\Big)
  \le\sum_{\ell=1}^{d}O(b_n)=O(b_n).
\]
Since $q\ge 3$, $b_n=n^{-(q-2)/2}(\log n)^{-q/2}=o(n^{-1/2})$, which is the claim.
\hfill$\square$

\paragraph*{Proof of Lemma~\ref{lem:pbHessianBound}}
Set $s=q+\delta$ and $r=q(q+\delta)/\delta$, so $\tfrac1q=\tfrac1r+\tfrac1s$.

\medskip\noindent\emph{Step 1 ($\lambda_0>0$).}
The matrices $A_0,A_n$ are symmetric, and for $u\in\R^{d}$,
$u^{\top}A_0u=\E\big[g'(\eta_1(\theta_0))\,(u^{\top}Z_0)^{2}\big]$. Since $g'>0$,
the integrand is nonnegative and vanishes only where $u^{\top}Z_0=0$; hence
$u^{\top}A_0u=0$ forces $u^{\top}Z_0=0$ a.s., and
Assumption~\ref{ass:pdHessian} gives $u=0$. Thus $A_0\succ0$ and $\lambda_0>0$.

\medskip\noindent\emph{Step 2 (reduction to the entries).}
By Weyl's inequality, $\lambda_{\min}(A_0)\le\lambda_{\min}(A_n)+\|A_n-A_0\|_{\mathrm{op}}$,
so $\{\lambda_{\min}(A_n)\le\lambda_0/2\}\subseteq\{\|A_n-A_0\|_{\mathrm{op}}\ge\lambda_0/2\}$.
Using $\|M\|_{\mathrm{op}}\le\|M\|_{F}\le d\max_{j,k}|M_{jk}|$ with $\tau:=\lambda_0/(2d)$,
\begin{equation}\label{eq:hess-union}
  \P\big(\lambda_{\min}(A_n)\le\lambda_0/2\big)
  \le\P\Big(\max_{1\le j,k\le d}\big|(A_n-A_0)_{jk}\big|\ge\tau\Big)
  \le\sum_{j,k=1}^{d}\P\big(|(A_n-A_0)_{jk}|\ge\tau\big).
\end{equation}

\medskip\noindent\emph{Step 3 ($q$-stability of each entry).}
Fix $(j,k)$ and set $W_t:=\wt K(\bV_t)_{jk}-(A_0)_{jk}$, so that
$(A_n-A_0)_{jk}=\frac1n\sum_{t=1}^{n}W_t$. The scalar map
$\bV_t\mapsto\wt K(\bV_t)_{jk}$ satisfies \eqref{eq:condLipsPoy} with $\alpha=2$,
because \eqref{eq:grad-score} yields
$|\wt K(\bV_t)_{jk}-\wt K(\wt\bV_t)_{jk}|\le\|\wt K(\bV_t)-\wt K(\wt\bV_t)\|
\lesssim(1+\|\bV_t\|^{2}+\|\wt\bV_t\|^{2})\|\bV_t-\wt\bV_t\|$. The required
moment $\E\|\bV_1\|^{2r}<\infty$ holds by Assumption~\ref{ass:stationaryDependence},
which provides a finite moment of order $2q(q+\delta)/\delta=2r$; moreover
$\Theta_{m,s}(\bV)\lesssim\rho^{m}$ (as in the proof of
Lemma~\ref{lem:controlScore}). Lemma~\ref{lem:fdm-transfer} with $\alpha=2$ then
gives $\Theta_{m,q}(\wt K_{jk})\le\bar C\,\Theta_{m,s}(\bV)\lesssim\rho^{m}$.
Centering does not affect the functional dependence coefficients, so
$(W_t)$ is stationary, centered, scalar, causal and $q$-stable with
$\Theta_{m,q}(W)\lesssim\rho^{m}$; its $q$-th moment is finite since
$\E|\wt K(\bV_0)_{jk}|^{q}\le\|g'\|_\infty^{q}\,\E\|Z_{-1}\|^{2q}
\lesssim 1+\E\|\bV_0\|^{2q}<\infty$, using $2q\le2r$.

\medskip\noindent\emph{Step 4 (per-entry rate).}
By Step 3, $(W_t)$ is a stationary, centered, scalar, $q$-stable process with
$\Theta_{m,q}(W)\lesssim\rho^{m}$ and $\E|W_0|^{q}<\infty$, so
Corollary~\ref{cor:nagaev-rate} applies to it. For $n$ large enough that
$a\sqrt{\log n/n}\le\tau$, monotonicity of $x\mapsto\P(|\frac1n\sum_{t\le n}W_t|\ge x)$
and Lemma ~\ref{cor:nagaev-rate} give
\[
  \P\big(|(A_n-A_0)_{jk}|\ge\tau\big)
  =\P\Big(\big|\tfrac1n\textstyle\sum_{t=1}^{n}W_t\big|\ge\tau\Big)
  \le\P\Big(\big|\tfrac1n\textstyle\sum_{t=1}^{n}W_t\big|\ge a\sqrt{\tfrac{\log n}{n}}\Big)
  =O(b_n)=o(n^{-1/2}),
\]
where $b_n=n^{-(q-2)/2}(\log n)^{-q/2}$ and $q\ge3$.

\medskip\noindent\emph{Step 5 (union bound).}
Summing the $d^{2}$ entrywise bounds in \eqref{eq:hess-union},
$\P(\lambda_{\min}(A_n)\le\lambda_0/2)\le d^{2}\,o(n^{-1/2})=o(n^{-1/2})$, which is
the claim.
\hfill$\square$

\paragraph*{Proof of Lemma~\ref{lem:controlRn}}
Set $s=q+\delta$ and $r=q(q+\delta)/\delta$, so $\tfrac1q=\tfrac1r+\tfrac1s$.

\medskip\noindent\emph{Step 1 (deterministic bound \eqref{eq:Rn-det}).}
Since $\eta_t(\theta)=\theta^{\top}Z_{t-1}$, the remainder is
$R_n(h)=\frac1n\sum_{t=1}^{n}c_t(h)\,Z_{t-1}\big(Z_{t-1}^{\top}h\big)$ with the
scalar $c_t(h)=\int_0^1\big[g'(\eta_t(\theta_0+sh))-g'(\eta_t(\theta_0))\big]\,ds$.
By the Lipschitz property of $g'$ and $\eta_t(\theta_0+sh)-\eta_t(\theta_0)=s\,h^{\top}Z_{t-1}$,
\[
  |c_t(h)|\le\int_0^1 C\,|s\,h^{\top}Z_{t-1}|\,ds=\tfrac{C}{2}\,|h^{\top}Z_{t-1}| .
\]
Hence, bounding one factor $|h^{\top}Z_{t-1}|$ by
$\|h\|\big(1+\sum_{i=1}^{p}|Y_{t-i}|+\|X_{t-1}\|\big)$ and the other by
$\|h\|\,\|Z_{t-1}\|$,
\[
  \|R_n(h)\|
  \le\frac1n\sum_{t=1}^{n}|c_t(h)|\,|Z_{t-1}^{\top}h|\,\|Z_{t-1}\|
  \le\frac{C}{2}\,\|h\|^{2}\,\frac1n\sum_{t=1}^{n}
     \Big(1+\sum_{i=1}^{p}|Y_{t-i}|+\|X_{t-1}\|\Big)\|Z_{t-1}\|^{2}
  =\frac{C}{2}\,\|h\|^{2}\,\bar R_n .
\]
Taking the supremum over $\|h\|\le r_n$ gives \eqref{eq:Rn-det}.

\medskip\noindent\emph{Step 2 (reduction).}
By \eqref{eq:Rn-det}, $\{\sup_{\|h\|\le r_n}\|R_n(h)\|\ge a_R r_n^{2}\}
\subseteq\{\bar R_n\ge 2a_R/C\}$, an event that no longer involves $r_n$; it
therefore suffices to show $\P(\bar R_n\ge 2a_R/C)=o(n^{-1/2})$.

\medskip\noindent\emph{Step 3 (limit and choice of $a_R$).}
Because $\bar K(\bV_0)\le(1+\|\bV_0\|)^{3}$ is integrable under
Assumption~\ref{ass:stationaryDependence}, the ergodic theorem gives
$\bar R_n\to\mu:=\E\bar K(\bV_0)<\infty$ almost surely. Fix any $a_R>\tfrac{C}{2}\mu$
and set $\varepsilon:=2a_R/C-\mu>0$, so that
$\P(\bar R_n\ge 2a_R/C)=\P(\bar R_n-\mu\ge\varepsilon)\le\P(|\bar R_n-\mu|\ge\varepsilon)$.

\medskip\noindent\emph{Step 4 ($q$-stability of the summand).}
Let $W_t:=\bar K(\bV_t)-\mu$, a stationary, centered scalar process, with
$\bar R_n-\mu=\frac1n\sum_{t=1}^{n}W_t$. The map
$\bV_t\mapsto\bar K(\bV_t)$ satisfies \eqref{eq:condLipsPoy} with $\alpha=2$,
since \eqref{eq:error-score} gives
$|\bar K(\bV_t)-\bar K(\wt\bV_t)|\lesssim(1+\|\bV_t\|^{2}+\|\wt\bV_t\|^{2})\|\bV_t-\wt\bV_t\|$;
the required moment $\E\|\bV_1\|^{2r}<\infty$ holds by
Assumption~\ref{ass:stationaryDependence}, which provides a finite moment of
order $2q(q+\delta)/\delta=2r$, and $\Theta_{m,s}(\bV)\lesssim\rho^{m}$ (as in
the proof of Lemma~\ref{lem:controlScore}). Lemma~\ref{lem:fdm-transfer} with
$\alpha=2$ then gives
$\Theta_{m,q}(\bar K)\le\bar C\,\Theta_{m,s}(\bV)\lesssim\rho^{m}$, and centering
leaves the coefficients unchanged, so $\Theta_{m,q}(W)\lesssim\rho^{m}$. Its
$q$-th moment is finite, $\E|W_0|^{q}\lesssim 1+\E\|\bV_0\|^{3q}<\infty$, since
$3q\le 2r$ by $\delta\le 2q$ (Assumption~\ref{ass:stationaryDependence}).

\medskip\noindent\emph{Step 5 (Fuk--Nagaev at a fixed level).}
By Step 4, $(W_t)$ is stationary, centered, scalar and $q$-stable with
$\Theta_{m,q}(W)\lesssim\rho^{m}$ and $\E|W_0|^{q}<\infty$, so
Corollary~\ref{cor:nagaev-rate} applies. For $n$ large enough that
$a\sqrt{\log n/n}\le\varepsilon$, monotonicity and
Lemma ~\ref{cor:nagaev-rate} give
\[
  \P\big(|\bar R_n-\mu|\ge\varepsilon\big)
  =\P\Big(\big|\tfrac1n\textstyle\sum_{t=1}^{n}W_t\big|\ge\varepsilon\Big)
  \le\P\Big(\big|\tfrac1n\textstyle\sum_{t=1}^{n}W_t\big|\ge a\sqrt{\tfrac{\log n}{n}}\Big)
  =O(b_n)=o(n^{-1/2}),
\]
and by Steps 2--3 the claim follows.
\hfill$\square$

\paragraph*{Proof of the lemma \ref{lem:be-projection}}
Fix $u$, $\|u\|=1$, put $c_u=L_0^{-1/2}u$ (so $\|c_u\|\le\lambda_L^{-1/2}$),
and set the scalar stationary causal process
$U_t^{(u)}=c_u^\top K(\bV_t)=e_t\,(c_u^\top Z_{t-1})$, for which
$u^\top\sqrt n\,L_0^{-1/2}S_n=n^{-1/2}\sum_{t=1}^{n}U_t^{(u)}$. We verify
Assumption~2.1 of \citet{jirak2016berry} at $p=3$, uniformly in $u$.

\emph{(i) Mean and moment.} As $e_t$ is a martingale difference and $Z_{t-1}$ is
predictable, $\E[U_t^{(u)}]=0$. Since $q\ge 3$, the third moment is finite:
$\|U_0^{(u)}\|_3\le\lambda_L^{-1/2}\big\|\,\|K(\bV_0)\|\,\big\|_3
\le\lambda_L^{-1/2}\big\|\,\|K(\bV_0)\|\,\big\|_q<\infty$, the last bound holding
by Assumption~\ref{ass:stationaryDependence} (as in
Lemma~\ref{lem:controlScore}).

\emph{(ii) Weak dependence.} The coupled version of $U_t^{(u)}$ is
$c_u^\top K(\wt\bV_t)$, so
$\theta_{3,l}(U^{(u)})\le\|c_u\|\,\theta_{3,l}(K)
\le\lambda_L^{-1/2}\theta_{q,l}(K)$, using $\|\cdot\|_3\le\|\cdot\|_q$. By
Lemma~\ref{lem:fdm-transfer} and the geometric decay in
Assumption~\ref{ass:stationaryDependence}, $\theta_{q,l}(K)\lesssim\rho^{l}$.
Hence
\[
  \Lambda(u):=\sum_{l\ge1}l^2\,\theta_{3,l}(U^{(u)})
  \le\lambda_L^{-1/2}\sum_{l\ge1}l^2\,\theta_{q,l}(K)=:D<\infty,
\]
a bound independent of $u$. 

\emph{(iii) Standardization.} As $(U_t^{(u)})$ is a martingale difference, its
long-run variance has no cross terms:
$s^2(u)=\E[(U_0^{(u)})^2]=c_u^\top L_0 c_u=u^\top u=1$, and likewise
$\big\|\sum_{t=1}^{n}U_t^{(u)}\big\|_2^2=\sum_{t=1}^{n}\E[(U_t^{(u)})^2]=n$, so
$s_n^2(u)=1$ exactly. Jirak's self-normalized statistic therefore equals the
projection.

 $\lambda_L>0$, so all of the above is well defined.
\citet[Theorem~2.2]{jirak2016berry} at $p=3$ then gives, for each $u$,
\[
  \sup_{x\in\R}\Big|\P\big(u^\top\sqrt n\,L_0^{-1/2}S_n\le x\big)-\Phi(x)\Big|
  \le\frac{B(\Lambda(u),s^2(u))}{n^{\,3/2-1}}
  =\frac{B(\Lambda(u),1)}{\sqrt n}.
\]
The constant depends on $u$ only through $(\Lambda(u),s^2(u))\in[0,D]\times\{1\}$,
a fixed bounded set; hence $B^\ast:=\sup_{\|u\|=1}B(\Lambda(u),1)<\infty$,
independent of $u$.
\hfill $\square$

\paragraph*{Proof of Lemma~\ref{lem:be-remainder}}
Write $r_n=a\sqrt{\log n/n}$ and $\lambda_L=\lambda_{\min}(L_0)>0$.

\emph{Step 0 (concentration of $L_n$).}
Each entry $(L_n-L_0)_{jk}=\frac1n\sum_{t}\big(e_t^2 Z_{t-1,j}Z_{t-1,k}-(L_0)_{jk}\big)$
is a centered average of a stationary scalar process. By the envelope
computation above ($\alpha=3$), Lemma~\ref{lem:fdm-transfer} under
Assumption~\ref{ass:moment3r} gives $\Theta_{m,q}\lesssim\rho^{m}$ and a finite
$q$-th moment; Corollary~\ref{cor:nagaev-rate} and a union bound over the $d^2$
entries then yield $a_L>0$ with
\[
  \P\big(\|L_n-L_0\|_{\mathrm{op}}\ge a_L\sqrt{\log n/n}\big)=o(n^{-1/2}).
\]
On the complementary event $\|L_n-L_0\|_{\mathrm{op}}<a_L\sqrt{\log n/n}\to0$, so
by Weyl $\lambda_{\min}(L_n)\ge\lambda_L/2$ for $n$ large.

\emph{Step 1 (good event).}
Let $\mathcal G_n$ be the intersection of the following events, each of
probability $1-o(n^{-1/2})$:
\begin{itemize}
\item $\|S_n\|\le a_S\sqrt{\log n/n}$ (Lemma~\ref{lem:controlScore});
\item $\widehat\theta_n$ is the unique solution of \eqref{eq:qmle} and
      $\|h_n\|\le r_n$ (Theorem~\ref{th:concentration});
\item $\sup_{\|h\|\le r_n}\|R_n(h)\|\le a_R r_n^{2}$ (Lemma~\ref{lem:controlRn});
\item $\lambda_{\min}(A_n)\ge\lambda_0/2$ and $\|A_n\|_{\mathrm{op}}\le 2\|A_0\|_{\mathrm{op}}$
      (Lemma~\ref{lem:pbHessianBound});
\item $\|L_n-L_0\|_{\mathrm{op}}\le a_L\sqrt{\log n/n}$ and
      $\lambda_{\min}(L_n)\ge\lambda_L/2$ (Step 0);
\item $\bar D_n^{A}:=\frac1n\sum_t\|Z_{t-1}\|^{3}\le 2\mu_A$ and
      $\bar D_n^{L}:=\frac1n\sum_t\|Z_{t-1}\|^{3}(1+\|\bV_t\|)\le 2\mu_L$,
      where $\mu_A,\mu_L<\infty$ are the corresponding means.
\end{itemize}
The last two are averages of degree-three and degree-four functionals of
$\bV_t$; under Assumption~\ref{ass:moment3r} they are $q$-stable with geometric
decay and finite $q$-th moment, so Lemma ~\ref{cor:nagaev-rate} at a fixed
level gives $o(n^{-1/2})$ for their complements. Hence
$\P(\mathcal G_n)=1-o(n^{-1/2})$. Note $\mathcal G_n$ does not depend on $u$.

Fix $\|u\|=1$ and work on $\mathcal G_n$. All bounds below use $\|u\|=1$ and are
operator-norm bounds, hence uniform in $u$.

\emph{Step 2 (term $T_1$).} With $\|M^{-1/2}-M_0^{-1/2}\|_{\mathrm{op}}\le
c(\lambda_L)\|M-M_0\|_{\mathrm{op}}$ on $\{M,M_0\succeq(\lambda_L/2)I\}$,
\[
  |T_1|\le c(\lambda_L)\,\|L_n-L_0\|_{\mathrm{op}}\,\sqrt n\|S_n\|
  \le c(\lambda_L)\,a_L\sqrt{\tfrac{\log n}{n}}\cdot a_S\sqrt{\log n}
  = c(\lambda_L)a_La_S\,\frac{\log n}{\sqrt n}.
\]

\emph{Step 3 (term $T_2$).} On $\{\lambda_{\min}(L_n)\ge\lambda_L/2\}$,
$\|L_n^{-1/2}\|_{\mathrm{op}}\le\sqrt{2/\lambda_L}$, and $\|h_n\|\le r_n$ gives
$\|R_n(h_n)\|\le\sup_{\|h\|\le r_n}\|R_n(h)\|\le a_R r_n^{2}$; hence
\[
  |T_2|\le\sqrt{2/\lambda_L}\,a_R r_n^{2}
  =\sqrt{2/\lambda_L}\,a_R a^2\,\frac{\log n}{n}=O\!\Big(\frac{\log n}{n}\Big).
\]

\emph{Step 4 (term $T_3$).} Decompose
$\widehat L_n^{-1/2}\widehat A_n-L_n^{-1/2}A_n
=\widehat L_n^{-1/2}(\widehat A_n-A_n)+(\widehat L_n^{-1/2}-L_n^{-1/2})A_n$.
By $g'$ Lipschitz and $\eta_t(\widehat\theta_n)-\eta_t(\theta_0)=h_n^\top Z_{t-1}$,
\[
  \|\widehat A_n-A_n\|_{\mathrm{op}}
  \le C\|h_n\|\,\bar D_n^{A}\le 2C\mu_A\|h_n\|,
\]
and similarly, by $g$ Lipschitz,
$\|\widehat L_n-L_n\|_{\mathrm{op}}\le C\|h_n\|\,\bar D_n^{L}\le 2C\mu_L\|h_n\|$;
since $\|h_n\|\le r_n\to0$, this gives $\lambda_{\min}(\widehat L_n)\ge\lambda_L/4$
and $\|\widehat L_n^{-1/2}\|_{\mathrm{op}}\le 2/\sqrt{\lambda_L}$, and
$\|\widehat L_n^{-1/2}-L_n^{-1/2}\|_{\mathrm{op}}\le c(\lambda_L)\,2C\mu_L\|h_n\|$
for $n$ large. With $\|A_n\|_{\mathrm{op}}\le 2\|A_0\|_{\mathrm{op}}$,
\[
  \|\widehat L_n^{-1/2}\widehat A_n-L_n^{-1/2}A_n\|_{\mathrm{op}}
  \le\Big(\tfrac{2}{\sqrt{\lambda_L}}\,2C\mu_A
        +c(\lambda_L)2C\mu_L\cdot 2\|A_0\|_{\mathrm{op}}\Big)\|h_n\|
  =:C_3\|h_n\|,
\]
so $|T_3|\le C_3\|h_n\|^{2}\le C_3 a^2\,\dfrac{\log n}{n}=O\!\Big(\dfrac{\log n}{n}\Big).$

\emph{Step 5 (conclusion).}
On $\mathcal G_n$, for every $\|u\|=1$ and $n$ large,
\[
  |\widetilde R_n(u)|\le|T_1|+|T_2|+|T_3|
  \le c(\lambda_L)a_La_S\,\frac{\log n}{\sqrt n}+O\!\Big(\frac{\log n}{n}\Big)
  \le \widetilde a_R\,\frac{\log n}{\sqrt n},
\]
for $\widetilde a_R>c(\lambda_L)a_La_S$. As the right-hand bound is $u$-free,
$\sup_{\|u\|=1}|\widetilde R_n(u)|\le\widetilde a_R\log n/\sqrt n$ on
$\mathcal G_n$, and $\P(\mathcal G_n^c)=o(n^{-1/2})$.
\hfill$\square$

\paragraph*{Proof of Lemma~\ref{lem:approxVar}}
Let $\lambda_0=\lambda_{\min}(A_0)>0$ (where we also have    the analogous
$\lambda_L>0$; here it is Assumption~\ref{ass:pdHessian} with $g'>0$). On the
event $\mathcal G_n$ of Lemma~\ref{lem:be-remainder}, of probability
$1-o(n^{-1/2})$, we have simultaneously
\begin{gather*}
  \|A_n-A_0\|_{\op}\le a_A\sqrt{\tfrac{\log n}{n}},\qquad
  \|\widehat L_n-L_0\|_{\op}\le a_L'\sqrt{\tfrac{\log n}{n}},\\
  \|\widehat A_n-A_n\|_{\op}\le C\|h_n\|\,\tfrac1n\textstyle\sum_t\|Z_{t-1}\|^2
  \le 2C\mu_A'\|h_n\|\le 2C\mu_A' a\sqrt{\tfrac{\log n}{n}},
\end{gather*}
the first two from Lemmas~\ref{lem:pbHessianBound} and \ref{lem:be-remainder}
(Step~0), the third from $g'$ Lipschitz together with the ergodic control of
$\frac1n\sum_t\|Z_{t-1}\|^2$ under Assumption~\ref{ass:moment3r}, and
$\|h_n\|\le a\sqrt{\log n/n}$ from Theorem~\ref{th:concentration}. Hence
$\|\widehat A_n-A_0\|_{\op}\le a_A''\sqrt{\log n/n}\to0$ and, by Weyl,
$\lambda_{\min}(\widehat A_n)\ge\lambda_0/2$, so
$\|\widehat A_n^{-1}\|_{\op}\le 2/\lambda_0$ and
$\|\widehat A_n^{-1}-A_0^{-1}\|_{\op}\le\|\widehat A_n^{-1}\|_{\op}\|A_0^{-1}\|_{\op}
\|\widehat A_n-A_0\|_{\op}\le\tfrac{2}{\lambda_0^2}a_A''\sqrt{\log n/n}$.
Writing $M=\widehat A_n^{-1}\widehat L_n\widehat A_n^{-1}$,
$M_0=A_0^{-1}L_0A_0^{-1}$ and using the identity
$M-M_0=(\widehat A_n^{-1}-A_0^{-1})\widehat L_n\widehat A_n^{-1}
+A_0^{-1}(\widehat L_n-L_0)\widehat A_n^{-1}
+A_0^{-1}L_0(\widehat A_n^{-1}-A_0^{-1})$ together with the operator-norm bounds
$\|\widehat L_n\|_{\op}\le\|L_0\|_{\op}+a_L'\sqrt{\log n/n}\le 2\|L_0\|_{\op}$ and
the three displays above, the triangle inequality gives
$\|M-M_0\|_{\op}\le C\sqrt{\log n/n}$ on $\mathcal G_n$, which is the claim.
\hfill$\square$

\end{document}